\documentclass[a4paper,oneside,reqno,11pt]{amsart}
\usepackage[latin1]{inputenc}
\usepackage[english]{babel}
\usepackage[margin=3cm]{geometry}
\usepackage{amsmath,amssymb,amsthm}
\usepackage{enumerate}
\usepackage{graphicx}
\usepackage{color}
\newtheorem{theorem}{Theorem}[section]

\newtheorem{lemma}[theorem]{Lemma}
\newtheorem{proposition}{Proposition}

\theoremstyle{definition}
\newtheorem{definition}[theorem]{Definition}
\newtheorem{remark}{Remark}

\definecolor{darkred}{rgb}{0.9,0.1,0.1}

\newcommand{\GG}{\mathbb{G}}

\newcommand{\be}{\begin{equation}}
\newcommand{\ee}{\end{equation}}
\newcommand{\ba}{\begin{array}}
\newcommand{\ea}{\end{array}}

\title[Stationary convection-diffusion equation in an infinite cylinder]{Stationary convection-diffusion equation in an infinite cylinder}

\subjclass{Primary: ???; Secondary: ???}
\keywords{Convection-diffusion equation, elliptic operators in unbounded domains, infinite cylinder, stabilization at infinity, effective drift, solvability, Fredholm alternative.}

 \email{irina.pettersson@uit.no}
 \email{andrey.piatnitski@uit.no}

\thanks{}

\makeatletter
\g@addto@macro{\endabstract}{\@setabstract}
\newcommand{\authorfootnotes}{\renewcommand\thefootnote{\@fnsymbol\c@footnote}}%
\makeatother

\begin{document}
\maketitle


\begin{center}

  \normalsize
  \authorfootnotes
  Irina Pettersson \textsuperscript{1} and Andrey Piatnitski\textsuperscript{1}

  \par \bigskip

  \textsuperscript{1}UiT The Arctic University of Norway\\[3mm]
  \bigskip

\today
\end{center}

\begin{abstract}
We study the existence and uniqueness of a solution to a linear stationary convection-diffusion equation stated in an infinite cylinder, Neumann boundary condition being imposed on the boundary. We assume that the cylinder is a junction of two semi-infinite cylinders with two different periodic regimes. Depending on the direction of the effective convection in the two semi-infinite cylinders, we either get a unique solution, or one-parameter family of solutions, or even non-existence in the general case. In the latter case we provide necessary and sufficient conditions for the existence of a solution.
\end{abstract}

\section*{Introduction}


The paper deals with a stationary linear convection-diffusion equation in an infinite cylinder $\mathbb G=(-\infty,\infty)\times Q$ with a Lipschitz bounded domain $Q\subset\mathbb R^{d-1}$, at the cylinder boundary the Neumann condition being imposed. We assume that, except for a compact set in $\mathbb G$, the coefficients of the convection-diffusion operator are periodic in $x_1$ both in the left and in the right
half cylinder. These two periodic operators need not coincide. This problem reads
\begin{equation}
\label{eq:or-prob}
  \left\{
   \begin{array}{lcr}
   \displaystyle{
    - {\rm div} \, (a(x) \, \nabla \, u(x)) + b(x)\cdot\nabla u(x) = f(x),}
    \quad \hfill x \in \GG,
    \\
    \displaystyle{
    \ \ \,a(x)\nabla u(x)\cdot n = g(x),} \quad \hfill x \in \Sigma.
   \end{array}
  \right.
\end{equation}
Under uniform ellipticity assumptions we study if this problem has a bounded solution and
if such a solution is unique. Concerning the functions $f$ and $g$ we assume that they decay
fast enough as $|x_1|\to\infty$.
Following \cite{PaPi} one can introduce the so-called effective axial drifts $\bar b^{+}$ and $\bar b^{-}$  in the right and left halves of the cylinder, respectively. It turns out that the mentioned existence and uniqueness issues depend on the signs of $\bar b^{+}$ and $\bar b^{-}$ (both effective drifts can be positive, or negative, or zero).

The main result of the paper is summarised below.

If $\bar b^{+}<0$ and $\bar b^{-}>0$, then for any two constants $K^-$ and $K^+$ there is a solution of \eqref{eq:or-prob} that converges to $K^-$ as $x_1\to -\infty$ and to $K^+$ as $x_1\to +\infty$.

If  $\bar b^{+}\ge 0$ and $\bar b^{-}>0$ or $\bar b^{+}<0$ and $\bar b^{-}\le0$ then a bounded solution exists and is unique up to an additive constant.

The case  $\bar b^{+}\ge 0$ and $\bar b^{-}\le 0$ is more interesting. In this case a bounded solution need not exist.  We will show that in this case the problem adjoint to \eqref{eq:or-prob} has a bounded solution $p\in  C(\bar{\mathbb G})$, which is positive under proper normalization. Then problem  \eqref{eq:or-prob} has a bounded solution if and only if
\begin{align}
\label{ort-cond}
\int\limits_{\mathbb G}f(x)p(x)\,dx+\int\limits_{\Sigma}g(x)p(x)d\sigma=0.
\end{align}
A bounded solution in this case is unique up to an additive constant.

The qualitative behaviour of the function $p$ in the two semi-infinite cylinders varies depending on whether the effective drift in that cylinder is equal to zero or not. Namely, if $\bar b^{+}<0$ and $\bar b^{-}>0$, then $p$ decays exponentially as $x_1 \to \infty$. If, however, the effective drift is zero in one of the semi-infinite cylinders, $p$ will stabilise to a periodic regime in that part, as $x_1 \to \infty$.

In all three cases any bounded solution converges
to some constants as $|x_1|\to~\infty$.
Moreover, this convergence has exponential rate if $f(x)$ and $g(x)$ decay exponentially as $|x_1|\to\infty$.

The question of the behavior at infinity of solutions to elliptic equations in cylindrical
and conical domains attracted the attention of mathematicians since the middle of
20th century.  In \cite{LaPa79} for a divergence form elliptic operator in a semi-infinite cylinder
there is a unique (up to an additive constant) bounded solution. It stabilizes to a constant at infinity.
Similar problem for a convection-diffusion operator has been studied in \cite{Pia84}, \cite{PaPi}. In these works necessary and sufficient conditions for the uniqueness of a bounded solution were provided.
In \cite{FlKn92}, \cite{FKP92} and \cite{KnPa96} specific classes of semi-linear elliptic equations in a half-cylinder
were considered. It was shown in particular that a global solution, if it exists, decays at
least exponentially with large axial distance.
The behavior at infinity of solutions to some classes of elliptic systems, in particular
to linear elasticity was investigated in \cite{KnLu98}.
In \cite{OlYo77} the uniqueness issue was studied for solutions of second order elliptic equations
in unbounded domains under some dissipation type assumptions on the coefficients.
 The work \cite{KoOl85} deals with solutions of elliptic systems in a cylinder that  have a bounded
 weighted Dirichlet integral.
Paper \cite{Naz99} studies the existence of solutions of symmetric elliptic systems in weighted spaces with exponentially growing or decaying weights.

In \cite{ChMar-2015} the authors study a Neumann problem for a linear elliptic operator in divergence form in a growing family of finite cylinders. It has been proved that the solution of this problem converges to a unique solution of a Neumann problem in the infinite cylinder.

In \cite{KoOl96} nonlinear elliptic equations  with a dissipative
nonlinear zero order terms was studied in a half-cylinder.
Nonlinear elliptic equations in unbounded domains, solvability and qualitative properties of the solutions have been considered in \cite{DaCiGiPuel-2004}, \cite{BeCaNi-1997-1}, \cite{BeCaNi-1997-2}. Fredholm theory of elliptic problems in unbounded domains is presented in \cite{Volpert-2011}.

To our best knowledge the question of existence and uniqueness of a bounded
solution to a convection-diffusion equation in an infinite cylinder has not been addressed
in the existing literature.

The paper is organized as follows. In Section \ref{s_1} we state the problem and provide all the assumptions.
Section \ref{s 2} deals with the case $\bar b^+<0$ and $\bar b^->0$. The main result here is
Theorem \ref{Th.1} that states that for any two constants $K^+$ and $K^-$ there is a solution that
stabilizes exponentially to $K^\pm$ as $x_1\to \pm\infty$.
In Section \ref{s_3} we consider the cases $\bar b^-\le 0$ and $\bar b^+<0$ and $\bar b^+\ge0$ and $\bar b^->0$.
The main result here is the existence and uniqueness up to an additive constant of a bounded solution,
see Theorem \ref{Th.2}. It is also shown that this solution stabilizes at infinity to some constants
at exponential rate.

Section \ref{s_3} focuses on the case $\bar b^+>0$ and $\bar b^-<0$. We first prove that the homogeneous adjoint problem has a localized solution in $H^1(\mathbb G)$, see Theorem \ref{Th-main-P(x)}.  Then we prove that
problem \eqref{eq:or-prob} has a bounded solution if and only if the orthogonality condition \eqref{ort-cond} is fulfilled. This is the subject of Theorem \ref{Th.3}.

In Section \ref{s_4} we study the remaining cases: $\bar b^-<0$ and $\bar b^+=0$ ($\bar b^-=0$ and $\bar b^->0$) and $\bar b^+=0$ and $\bar b^-=0$. We show that a bounded solution to \eqref{eq:or-prob} exists if and only if the orthogonality condition \eqref{ort-cond} is satisfied, where $p$ is a function from the kernel of the adjoint operator which decays exponentially in the half cylinder  if the corresponding effective drift is not equal to zero, and stabilises to a periodic regime in the half cylinder where the effective drift is zero.

For the reader convenience, in Section~\ref{s_5} we summarize the results of \cite{PaPi} which we use throughout the paper.


\section{Problem statement}
\label{s_1}

Given a bounded domain $Q \subset \mathbb{R}^{d-1}$ with a Lipschitz boundary $\partial Q$, we denote by $\GG$ an infinite cylinder $\mathbb{R} \times Q$ with points $x = (x_1, x')$ and the axis directed along $x_1$. The lateral boundary of the cylinder is denoted by $\Sigma = \mathbb{R} \times \partial Q$. We study the following Neumann boundary value problem for a stationary convection-diffusion equation:
\begin{equation}
\label{eq:orig-prob}
  \left\{
   \begin{array}{lcr}
   \displaystyle{
    A \, u \equiv - {\rm div} \, (a(x) \, \nabla \, u(x)) + b(x)\cdot\nabla u(x) = f(x),}
    \quad \hfill \mbox{in}\,\, \GG,
    \\
    \displaystyle{
    B \, u \equiv a(x)\nabla u(x)\cdot n = g(x),} \quad \hfill \mbox{on}\,\, \Sigma.
   \end{array}
  \right.
\end{equation}
Here $v\cdot w = \sum_{i=1}^d v_i w_i$, $v, w \in \mathbb{R}^d$, denotes the standard scalar product in $\mathbb{R}^d$; $n$ is the exterior unit normal.

\begin{definition}
We say that a solution $u$ of problem \eqref{eq:orig-prob} is bounded if
\[
\|u\|_{L^2(G_N^{N+1})} \le C, \quad \forall N.
\]
The goal of the paper is to study the question of existence and uniqueness of a bounded solution to problem \eqref{eq:orig-prob}.
\end{definition}

Throughout the paper we use the notations
$$
G_\alpha^\beta = (\alpha, \beta) \times Q, \quad
\Sigma_\alpha^\beta = (\alpha, \beta) \times \partial Q,
\quad S_\alpha = \{\alpha\} \times Q.
$$
Our main assumptions:
\begin{itemize}
\item[\bf (H1)]
The coefficients $a_{ij}, b_j \in L^\infty(\GG)$ are periodic in $x_1$ outside the finite cylinder $G_{-1}^1$, that is
$$
a_{ij}(x) = \left\{
\begin{array}{c}
a_{ij}^{+}(x), \,\,x \in G_1^{+\infty},
\\[1mm]
\tilde{a}_{ij}(x), \,\, x \in G_{-1}^1,
\\[1mm]
a_{ij}^{-}(x), \,\,x \in G_{-\infty}^{-1};
\end{array}
\right.
\qquad
b_{j}(x) = \left\{
\begin{array}{c}
b_{j}^{+}(x), \,\,x \in G_1^{+\infty},
\\[1mm]
\tilde{b}_{j}(x), \,\, x \in G_{-1}^1,
\\[1mm]
b_{j}^{-}(x), \,\,x \in G_{-\infty}^{-1},
\end{array}
\right.
$$
where $a_{ij}^{+}, b_j^{+}$ and $a_{ij}^{-}, b_j^{-}$ are $1$-periodic with respect to $x_1$ in $G_1^{+\infty}$ and $G_{-\infty}^{-1}$, respectively.

\item[\bf(H2)]
The $d\times d$ matrix $a(x)$ is symmetric and satisfies the uniform ellipticity condition,
that is there exists a positive constant $\Lambda$ such that, for almost all $x \in \GG$,
\begin{equation}
\label{1.2}
    a(x) \, \xi \cdot \xi \ge \Lambda \, |\xi|^2,
    \quad  \forall \xi \in \mathbb{R}^d.
\end{equation}
\item[\bf (H3)]
The functions $f(x) \in L^2(\GG)$ and $g(x) \in L^2(\Sigma)$ decay exponentially to zero as $|x_1| \to \infty$. Namely, there exist positive constants $C_0, \gamma_0$ independent of $n$ such that
$$
\|f\|_{L^2(G_n^{n+1})}+
\|g\|_{L^2(\Sigma_n^{n+1})} \leq C_0 \, e^{-\gamma_0 \, |n|}, \,\, n \in \mathbb R.
$$
\end{itemize}

The presence of two periodic regimes in the two semi-infinite parts of the cylinder, $G_{-\infty}^0, G_0^{+\infty}$ makes the problem nontrivial.

The existence and uniqueness issue depends on the signs of the effective convection (effective drift) in the half-cylinders $G_{-\infty}^0$ and $G_0^{+\infty}$. The effective convection in the direction of $x_1$ for each periodic regime, $a_{ij}^+, b_j^+$ and $a_{ij}^-, b_j^-$, is defined as follows:
\begin{align}
\label{bar{b}}
\bar{b}^{\pm} = \int \limits_{Y} ( a_{1j}^\pm(x) \partial_j p^\pm(x) + b_1^\pm(x) \, p^\pm(x)) \, dx,
\end{align}
where $Y= \mathbb{T}^1 \times Q$ is the periodicity cell, $\mathbb{T}^1$ is a one-dimensional torus, and $p^\pm(y)$ belong to the kernels of adjoint periodic operators
\begin{equation}
\label{eq:p-peridic}
\left\{
\begin{array}{lcr}
- {\rm div}\, ( a^\pm \nabla \, p^\pm) -
{\rm div}\, ( b^\pm \, p^\pm) = 0, \quad \hfill y \in Y,
\\[1.5mm]
\displaystyle
a^\pm\nabla p^\pm\cdot n + (b^\pm\cdot n)  p^\pm = 0, \quad \hfill y \in \partial Y.
\end{array}
\right.
\end{equation}
Each of problems \eqref{eq:p-peridic} has a unique up to a multiplicative constant solution $p^\pm \in H^1(Y)\cap C(\overline Y)$ which is positive everywhere in $\overline{Y}$ (see, for example, \cite{PaPi}, Section~2).

The existence and the properties of solutions of problem \eqref{eq:orig-prob} depend crucially on the signs of $\bar{b}^{+}$ and $\bar{b}^{-}$. We are going to study problem \eqref{eq:orig-prob} for all possible combinations of signs of the effective drift in the two semi-cylinders:
\begin{enumerate}
\item
$\bar{b}^{+} < 0$, $\bar{b}^{-} > 0$ (two-parameter family of solutions to \eqref{eq:orig-prob});
\item
$\bar{b}^{+} < 0$, $\bar{b}^{-} \le 0$ (or $\bar{b}^{+} > 0$, $\bar{b}^{-} > 0$) (one-parameter family of solutions);
\item
$\bar{b}^{+} \ge 0$, $\bar{b}^{-} \le 0$ (non-existence in general case).
\end{enumerate}

\section{Case $\bar{b}^{+}<0, \quad \bar{b}^{-}>0$}
\label{s 2}

\begin{theorem}
\label{Th.1}
Let conditions $\bf(H1)-(H3)$ be fulfilled and suppose that $\bar{b}^{+}< 0$ and $\bar{b}^{-}> 0$. Then, for any constants $K^{+}$ and $K^{-}$, there exists a unique bounded solution $u(x)$ of problem \eqref{eq:orig-prob} that converges at exponential rate for some $\gamma >0$ to these constants, as $x_1 \to \pm \infty$:
\begin{equation}
\label{1.3}
\begin{array}{l}
\displaystyle
\|u - K^{-}\|_{L^2(G_{-\infty}^{-n})} + \|u - K^{+}\|_{L^2(G_n^{+\infty})} +
\|\nabla u\|_{L^2(\mathbb{G} \setminus G_{-n}^{n})} \le C\, M_1  \, e^{- \gamma \, n},
\\[2mm]
\displaystyle
\|\nabla u\|_{L^2(\GG)} \leq C\, M_1, \quad n \in \mathbb R_+.
\end{array}
\end{equation}
The constant $M_1$ in \eqref{1.3} have the form
\[
M_1= |K^{+} - K^{-}| + \|(1 + x_1^2)\,f\|_{L^2(\GG)} +
 \|(1 + x_1^2)\,g\|_{L^2(\Sigma)},
\]
where $C$ depends on $\Lambda, d$ and $Q$.

\end{theorem}

\begin{proof}

Let us note that any bounded solution in $\mathbb G$ restricted to the left or right semi-infinite cylinder is a bounded solution there. Thus, by Theorem~\ref{PP-Th-semi-inf}, we conclude that every bounded solution (if it exists) stabilizes to some constants at $x_1 \to \pm \infty$.

Due to the linearity of problem \eqref{eq:orig-prob}, we can consider the homogeneous ($f=g=0$) and nonhomogeneous equations separately. At the first step we prove the existence of a solution to the homogeneous equation that stabilizes to some nonzero constants as $|x_1|\to \infty$. In the second step we show that there exists $u$ that solves the nonhomogeneous problem \eqref{eq:orig-prob} and decays to zero as $|x_1| \to \infty$.
$ $\\ \textbf{ The case $f=g=0$.}
For two arbitrary constants $K^{+}, K^{-} \in \mathbf R$ and $k \in \mathbb R_{+}$, we consider the following sequence of the auxiliary boundary value problems:
\begin{align}
\label{2.2}
  \left\{
   \begin{array}{lcr}
   \displaystyle{
    A \, u^k = 0,}
    \quad \hfill x \in G_{-k}^k,
    \\[1mm]
    \displaystyle{
    B \, u^k = 0,} \quad \hfill x \in \Sigma_{- k}^k,
    \\[1mm]
    \displaystyle
    u^k( \pm k, x') = K^{\pm},\quad \hfill x' \in Q.
   \end{array}
  \right.
\end{align}
We assume that $K^{+} \neq K^{-}$, otherwise the result of the theorem is trivial: $u \equiv K^{+}$. Without loss of generality we assume that $K^{+}> K^{-}$. Denote $\displaystyle v^k=u^k - \frac{K^{+}+K^{-}}{2}$. Then $v^k$ solves the problem
\begin{align}
\label{eq:v^k}
\left\{
   \begin{array}{lcr}
   \displaystyle{
    A \, v^k = 0,}
    \quad \hfill x \in G_{-k}^k,
    \\[1mm]
    \displaystyle{
    B \, v^k = 0,} \quad \hfill x \in \Sigma_{- k}^k,
    \\[1mm]
    \displaystyle
    v^k( \pm k, x') = \pm \, \frac{1}{2} \, (K^{+} -K^{-}), \quad \hfill x' \in Q.
   \end{array}
\right.
\end{align}
By the maximum principle,
\begin{equation}
\label{2.3}
|v^k| \leq \frac{1}{2}\, |K^{+} -K^{-}|, \quad x \in G_{-k}^k, \quad k \in \mathbb R_{+}.
\end{equation}
Indeed, by the maximum principle, a negative minimum cannot be attained in the interior of the domain $G_{-k}^k$. The assumption that a negative minimum is attained on the lateral boundary $\Sigma_{-k}^k$ also contradicts the maximum principle. Indeed, one can prove this extending $v_k$ by reflection across the lateral boundary and using the fact that $v_k$ satisfies homogeneous Neumann boundary condition on $\Sigma_{-k}^k$. This argument is used many times throughout the paper and allows us to apply the maximum principle, the Harnack inequality and Nash estimates up to the lateral boundary of the cylinder.

It follows directly from \eqref{2.3} that the $L^2(G_{N}^{N+1})$-norm of $v^k$ is bounded, and by the elliptic estimates (see \cite{GilTrud}, Ch.8, problem 8.2), the norm of $\nabla v^k$ is also bounded in each finite cylinder:
\[
\|v^k\|_{L^2(G_{N}^{N+1})} + \|\nabla v^k\|_{L^2(G_{N}^{N+1})} \le C\, |K^{+} -K^{-}|, \quad N \in \mathbb R,
\]
with $C$ independent of $N$. Here we extend $v^k$ by constants $\pm (K^+ - K^-)/2$ outside $G_{-k}^k$.
Consequently, we obtain
\begin{align}
\label{2.4}
\|u^k\|_{L^2(G_N^{N+1})} & \le C \, \big( |K^{+}| +  |K^{-}| \big),
\\
\|\nabla \, u^k\|_{L^2(G_N^{N+1})} & = \|\nabla \, v^k\|_{L^2(G_N^{N+1})}
\leq C \, |K^{+} -K^{-}|, \,\, N \in \mathbb R,
\end{align}
where the constant $C$ depends only on $\Lambda, d$ and $Q$. By the compactness of embedding $H^1(G_\alpha^\beta) \Subset L^2(G_\alpha^\beta)$, we conclude that, up to a subsequence, $u^k$ converges to a solution $u$ to problem \eqref{eq:orig-prob} (with $f=g=0$) strongly in $L_{loc}^2(\GG)$ and $\nabla \, u^k \rightharpoonup \nabla\, u$ weakly in $(L_{loc}^2(\GG))^d$, as $k \to \infty$. This proves the existence of a solution $u \in H_{loc}^1(\GG)$ to \eqref{eq:orig-prob}.

Note that the H\"{o}lder norm of $u^k$ in each cylinder of fixed length is bounded (see \cite{GilTrud}, Theorem 8.24):
\begin{align}
\label{1.8}
\|u^k\|_{C^\alpha(\overline{G_{N}^{N+1}})} \leq C \|u^k\|_{L^2(G_{N-1}^{N+2})} \le
C \,( |K^{+}| + |K^{-}|), \quad \forall N,
\end{align}
with $\alpha > 0$ and a constant $C$ depending only on $d$, $Q$ and $\Lambda$.


Due to \eqref{1.8}, $u^k$ converges to $u$ uniformly in each finite cylinder $G_N^{N+1}$, as $k \to \infty$, and
\[
|u| \le C \,( |K^{+}| + |K^{-}|), \quad x \in \GG.
\]

We proceed with the exponential stabilization of $u$, as $x_1 \to \infty$.

Let us compare the solution $v^k$ of \eqref{eq:v^k} with a solution $\hat{v}^k $ to the following problem in the semi-infinite cylinder
\begin{align}
\label{2.5}
  \left\{
   \begin{array}{lcr}
   \displaystyle{
    A \, \hat{v}^k = 0,}
    \quad \hfill x \in G_{1}^k,
    \\[1mm]
    \displaystyle{
    B \, \hat{v}^k = 0,} \quad \hfill x \in \Sigma_1^k,
    \\[1mm]
    \displaystyle
    \hat{v}^k( 0, x') = - |K^{+}-K^{-}|/2, \quad \hat{v}^k(k, x') = (K^{+}-K^{-})/2, \quad \hfill x' \in Q.
   \end{array}
  \right.
\end{align}
By the maximum principle, $v^k \geq \hat{v}^k$  and $\hat{v}^k - \frac{K^{+}-K^{-}}{2} < 0$ in $G_0^k$. By Theorem~\ref{PP-Th-semi-inf}, in the case $\bar{b}^{+} < 0$, the following estimate is valid:
\[
|u^k - K^{+}|=|v^k - \frac{K^{+}-K^{-}}{2}| \leq |\hat{v}^k - \frac{K^{+}-K^{-}}{2}| \leq C\, |K^{+} - K^{-}| \, e^{- \gamma x_1}, \quad x \in G_1^k.
\]
Since, up to a subsequence, $\{u^k\}$ converges to $u$ uniformly on every compact set $K \subset \GG$, then
\[
|u - K^{+}| \leq C\, |K^{+} - K^{-}| \, e^{- \gamma x_1}, \quad x\in G_1^{+\infty}.
\]
The last estimate yields
\[
\|u-K^{+}\|_{L^2(G_N^{N+1})} \le C\,  |K^{+} - K^{-}|\, e^{- \gamma N}, \quad N=1,...,k-1.
\]
By the elliptic estimates we obtain
$$
\|\nabla u\|_{L^2(G_N^{N+1})} \le C \|u-K^{+}\|_{L^2(G_{N-1}^{N+2})} \le C\,  |K^{+} - K^{-}|\, e^{- \gamma N}, \quad N=1,...,k-1.
$$
The convergence of $u$ to $K^{-}$, as $x_1 \to -\infty$, is proved in the same way.


\textbf{The case when at least one of the functions $f$ or $g$ not zero.}

We prove the existence of a solution of the nonhomogeneous problem \eqref{eq:orig-prob} that decays exponentially at infinity. To this end we consider the following problems:
\begin{equation}
\label{1.11}
  \left\{
   \begin{array}{lcr}
   \displaystyle{
    A\, u_k =  f(x),}
    \quad \hfill x \in G_{-k}^k,
    \\[1mm]
    \displaystyle{
    B\, u_k = g(x),} \quad \hfill x \in \Sigma_{- k}^k,
    \\[1mm]
    \displaystyle{
    u_k( - k, x') = u_k(k, x') = 0}, \quad \hfill x' \in Q.
   \end{array}
  \right.
\end{equation}
Without loss of generality we assume that $f(x)\ge 0$ and $g(x)\ge 0$, otherwise we represent these functions as two sums of their positive and negative parts. Moreover, we assume that ${\rm supp}\, f, {\rm supp}\,g \subset G_0^{+\infty}$. The case when the supports of $f$ and $g$ are in $G_{-\infty}^0$ can be considered similarly.

Suppose first that the coefficients $a_{ij}, b_j$ and the functions $f$ and $g$ are smooth. Thus, by the strong maximum principle (see, for example, \cite{GilTrud}), $u^k(x) > 0$, $x \in G_{-k}^k \cup \Sigma_{-k}^k$.

Due to Lemma~\ref{PP-Th-v^k-1}, in the semi-infinite cylinder $G_{-\infty}^{-1}$, where $\bar b^- > 0$, $u_k$ decays exponentially and the following estimate holds:
$$
u_k(x_1, x') \leq C_0 \, \|u_k\|_{L^{\infty}(S_{-1})} \, e^{\gamma\, x_1}, \quad x_1 < -1, \,\, \gamma>0,
$$
where $C_0$ depends only on $\Lambda, d$ and $Q$.
Since $u_k>0$, by the Harnack inequality, there exists $\alpha$ which depends only on $d, Q$ and $\Lambda$ such that
$$
u_k(x)\leq \alpha \, e^{\gamma \, x_1} \, \min \limits_{\overline G_{-1}^0} u_k(x), \quad x \in G_{-\infty}^{-1}.
$$
Obviously, there exists $\xi > 1$ such that
\begin{equation}
\label{1.12}
u_k(-\xi, x') < \frac{|Q|}{2}\, \min \limits_{\overline G_{-1}^0} u_k(x).
\end{equation}

Due to the linearity of the problem in $G_{-\xi}^k$ we represent $u_k$ as a sum $v_k + w_k$, where $v_k$ is a solution of the homogeneous equation with nonzero Dirichlet boundary conditions
\begin{equation}
\label{1.13}
  \left\{
   \begin{array}{lcr}
   \displaystyle{
    A\, v_k =  0,}
    \quad \hfill x \in G_{-\xi}^k,
    \\[1mm]
    \displaystyle{
    B\, v_k = 0,} \quad \hfill x \in \Sigma_{- \xi}^k,
    \\[1mm]
    \displaystyle{
    v_k( - \xi, x') = u_k( - \xi, x'), \,\, v_k(k,x') = 0}, \quad \hfill x' \in Q;
   \end{array}
  \right.
\end{equation}
and $w_k$ is a solution of the problem
\begin{equation}
\label{1.14}
  \left\{
   \begin{array}{lcr}
   \displaystyle{
    A\, w_k =  f(x),}
    \quad \hfill x \in G_{-\xi}^k,
    \\[1mm]
    \displaystyle{
    B\, w_k = g(x),} \quad \hfill x \in \Sigma_{- \xi}^k,
    \\[1mm]
    \displaystyle{
    w_k( - \xi, x') = w_k(k,x') = 0}, \quad \hfill x' \in Q.
   \end{array}
  \right.
\end{equation}
By the maximum principle we have
$$
v_k(x) \leq \frac{|Q|}{2} \, \min \limits_{\overline G_{-1}^0} u_k(x), \quad x \in G_{-\xi}^k.
$$
By Lemma~\ref{PP-Th-v^k-2}, a solution $w_k$ of problem \eqref{1.14} satisfies the following estimate:
$$
\|w_k\|_{L^2(G_N^{N+1})} \leq C \, \|(1 + x_1^2) \, f\|_{L^2(G_0^{+\infty})} + C \, \|(1 + x_1^2) \, g\|_{L^2(\Sigma_0^{+\infty})}.
$$
Thus,
\begin{align*}
|Q| \min \limits_{G_{-1}^0} u_k(x) & \leq \|u_k\|_{L^2(G_{-1}^0)}
\\
& \leq \|v_k\|_{L^2(G_{-1}^0)} + \|w_k\|_{L^2(G_{-1}^0)} \leq
\frac{|Q|}{2} \, \min \limits_{G_{-1}^0}u_k(x) + \|w_k\|_{L^2(G_{-1}^0)}.
\end{align*}
It follows from the last inequality that
\begin{equation}
\label{1.15}
\min \limits_{G_{-1}^0} u_k(x) \leq C ( \|(1 + x_1^2) \, f\|_{L^2(G_0^{+\infty})} +  \|(1 + x_1^2) \, g\|_{L^2(\Sigma_0^{+\infty})}),
\end{equation}
where $C = C(\Lambda, d, Q)$.
With the help of the Harnack inequality, maximum principle and \eqref{1.15} we get
\begin{align*}
|u_k(x)| & \le C ( \|(1 + x_1^2) \, f\|_{L^2(G_0^{+\infty})} +  \|(1 + x_1^2) \, g\|_{L^2(\Sigma_0^{+\infty})})\, e^{\gamma x_1}, \quad x \in G_{-k}^{-1}.
\end{align*}
Note that $u_k$ is smooth in $G_{-k}^{-1}$ where it solves a homogeneous problem.

It remains to apply Lemma~\ref{PP-Th-v^k-1} and Lemma~\ref{PP-Th-v^k-2}. According to these results, for $\bar{b}^{-} >0$ and $\bar{b}^{+} <0$, we obtain
\begin{align*}
\|u_k\|_{L^2(G_N^{N+1})} & \le C\, M, \,\, \forall N>0;
\\
\|u_k\|_{L^2(G_N^{N+1})} & \le C\, M \, e^{-\gamma N}, \,\, \forall N<0;
\\
\|\nabla u_k\|_{L^2(G_{-k}^k)} & \le C \, M,
\end{align*}
where the constant $M$ has the form
$$
M = \|(1 + x_1^2) \, f\|_{L^2(G_0^{+\infty})} + \|(1 + x_1^2) \, g\|_{L^2(\Sigma_0^{+\infty})}.
$$
For the nonsmooth data the desired estimates can be justified by means of a smoothening procedure.

Thus, one can see that, up to a subsequence, $\{u_k\}$ (being extended by zero to the whole cylinder $\GG$), converges weakly in $H_{loc}^1(\GG)$ to a solution $u$ of problem \eqref{eq:orig-prob}. Moreover, by Thereom~\ref{PP-Th-semi-inf}, $u$ stabilizes exponentially to some constants as $x_1 \to \pm \infty$. One can show that by construction $u$ actually decays exponentially, as $x_1 \to \pm \infty$, but it is of no importance at this stage.

As was shown above, for any constants $K^\pm$, there exists a solution of homogeneous equation, stabilizing to these constants at infinity and satisfying estimates \eqref{1.3}. Summing up such a solution with the particular solution $u(x)$ of the non-homogeneous equation, we obtain the desired solution of nonhomogeneous problem.

The uniqueness of a solution for fixed constants $K^\pm$ follows from the maximum principle. Assume that there exists $u$ solving \eqref{eq:orig-prob} with $f=g=0$ and $u \to 0$ as $x_1 \to \pm \infty$. Let us restrict $u$ on $G_{-k}^k$. Due to the exponential decay, $|u(\pm k, x')| \le C e^{-\gamma k}$ for any $k$. By the maximum principle, $|u| \le C e^{-\gamma k}$ everywhere in $G_{-k}^k$ for any $k$, which implies that $u=0$.

Theorem \ref{Th.1} is proved.

\end{proof}

\section{The case $\bar{b}^{-}\le 0, \,\, \bar{b}^{+} < 0$ ($\bar{b}^{-}> 0, \,\, \bar{b}^{+} \ge 0$).}
\label{s_3}

\begin{theorem}
\label{Th.2}
Suppose that conditions $\bf(H1)$--$\bf(H3)$ are fulfilled and $\bar{b}^{-}\le 0, \,\, \bar{b}^{+} < 0$ ($\bar{b}^{-}> 0, \,\, \bar{b}^{+} \ge 0$). Then there exists a unique, up to an additive constant, bounded solution $u(x)$ of problem \eqref{eq:orig-prob}. This solution, for some constant $K^{-}$, satisfies the bounds
\begin{align}
\label{eq:est-u-case2}
\|u - K^{-}\|_{L^2(G_{N-1}^{N})} & \le C\, M \, e^{\gamma \, N}, \quad N < 0, \notag
\\
\|u\|_{L^2(G_N^{N+1})} & \le C\, M \, e^{- \gamma \, N}, \quad N > 0,
\\
\|\nabla \, u\|_{L^2(\GG)} & \le  C M. \notag
\end{align}
Here the constant $M$ is given by
\begin{align}
\label{M}
M= \|(1 + x_1^2)\,f\|_{L^2(\GG)} +
 \|(1 + x_1^2)\,g\|_{L^2(\Sigma)},
\end{align}
and $C$ only depend on $\Lambda, d$ and $Q$.
\end{theorem}

\begin{remark}
Note that in the case  $\bar{b}^{-}\le 0, \,\, \bar{b}^{+} < 0$ there exists a unique, up to an additive constant, solution to problem \eqref{eq:orig-prob} with $f=g=0$ which is equal to zero (so as $K^{-}$). Indeed, the solution is unique by Theorem~\ref{Th.2} and $u=0$ is a solution.
\end{remark}
\begin{proof}
We will prove Theorem \ref{Th.2} in the case $\bar{b}^{-}\le 0, \,\, \bar{b}^{+} < 0$. The case $\bar{b}^{-}> 0, \,\, \bar{b}^{+} \ge 0$ is treated in a similar way.

We prove the existence of a solution to \eqref{eq:orig-prob} by considering the auxiliary problems in finite cylinders
\begin{align}
\label{eq:fin-cyl-case2}
  \left\{
   \begin{array}{lcr}
   \displaystyle{
    A\, u_k =  f,}
    \quad \hfill x \in G_{-k}^k,
    \\[1mm]
    \displaystyle{
    B\, u_k = g,} \quad \hfill x \in \Sigma_{-k}^k,
    \\[1mm]
    \displaystyle
    u_k(-k, x') = u_k(k, x') = 0, \quad \hfill x' \in Q.
   \end{array}
  \right.
\end{align}
If both $f$ and $g$ are equal to zero, the problem is trivial: $u=\rm const$ is a solution to \eqref{eq:orig-prob}. We focus on the case when at least one of these functions is not zero. Without loss of generality we can assume that $f, g \geq 0$. Otherwise we represent them as the sums of positive and negative parts and repeat the argument. In addition we assume that the coefficients of the equation $a_{ij}$, $b_j$, as well as the functions $f,\,g$ are smooth. The case of nonsmooth data is justified by means of smoothing. Then by the maximum principle $u_k > 0$ in $G_{-k}^k$ up to the lateral boundary.

We will consider two cases: ${\rm supp} f, {\rm supp} \,\, g \subset G_\eta^{+\infty}$ and ${\rm supp} \, f, {\rm supp}\,\,g \subset G_{-\infty}^\eta$ for some $\eta >0$ which will be chosen later.

\bigskip

Let now ${\rm supp}\: f, {\rm supp}\:g \subset G_{\eta}^{+\infty}$. To separate difficulties, as before, we represent the solution $u_k$ in the cylinder $G_0^k$ as the sum $u_k = v_k + w_k$, where $v_k$ and $w_k$ are solutions of the following problems:
\[
  \left\{
   \begin{array}{lcr}
   \displaystyle{
    A\, v_k = f(x),}
    \quad \hfill x \in G_{0}^k,
    \\[1mm]
    \displaystyle{
    B\,v_k = g(x),} \quad \hfill x \in \Sigma_0^k,
    \\[1mm]
    \displaystyle{
    v_k(0, x') = v_k(k, x') = 0}, \,\, \hfill x' \in Q;
   \end{array}
  \right.
\]
\[
  \left\{
   \begin{array}{lcr}
   \displaystyle{
    A\, w_k = 0,}
    \quad \hfill x \in G_{0}^k,
    \\[1mm]
    \displaystyle{
   B\,w_k = 0,} \quad \hfill x \in \Sigma_0^k,
    \\[1mm]
    \displaystyle{
    w_k(0, x') = u_k(0, x'), \,\, w_k(k, x') = 0}, \quad \hfill x' \in Q.
   \end{array}
  \right.
\]
Due to Lemma~\ref{PP-Th-v^k-2}, $v_k$ satisfies the following estimate:
\begin{align}
\label{3.4}
&\|v_k\|_{L^2(G_N^{N+1})}  + \|\nabla v_k\|_{L^2(G_N^{N+1})}\le C\, M, \quad N >0,
\end{align}
where $C = C(\Lambda, d, Q)$ and $M$ is given by \eqref{M}.

It is left to show that $\|u_k\|_{L^\infty(S_0)}$ is bounded. Then by the maximum principle it will follow immediately that $\|w_k\|_{L^\infty(G_0^k)}$ is bounded.

Since $\bar b^{+}<0$, $w_k$ decays exponentially with $x_1$, and for any $\delta$ we can choose $\eta = \eta(\delta)>0$ such that
\begin{align}
\label{3.5}
w_k(\eta,x') \leq  \delta \, \min \limits_{\overline Q} u_k(0,x').
\end{align}

On the other hand,
\[
\|u_k\|_{L^2(G_{\eta-1}^\eta)} \ge |Q| \min \limits_{\overline G_{\eta-1}^\eta} u_k.
\]
Since $u_k(-k,x')=0$ and $u_k$ solves a homogeneous problem in $G_{-k}^{\eta}$, then \\ $\min_{x' \in \overline Q}u_k(x_1,x')$ is an increasing function of $x_1$ on $(-k, \eta)$.

Indeed, $\min_{x' \in \overline Q}u_k(x_1,x')$ cannot attain a nonnegative minimum inside $G_{-k}^{\eta}$, which yields that it is either increasing or decreasing starting from some point ($\min_{x' \in \overline Q}u_k(x_1,x')$ might have one local maximum). But in the latter case $\max_{x' \in \overline Q}u_k(x_1,x')$ is also decreasing, which is impossible since $u_k(-k,x')=0$ and $u_k > 0$ in $G_{-k}^k$.

Thus
\[
\|u_k\|_{L^2(G_{\eta-1}^\eta)} \ge |Q| \min \limits_{\overline Q} u_k(0,x').
\]
Using \eqref{3.5}, the Harnack inequality and the maximum principle we obtain
\begin{align*}
|Q| \min \limits_{\overline Q} u_k(0,x')
\le \|u_k\|_{L^2(G_{\eta-1}^\eta)}
&\le \|v_k\|_{L^2(G_{\eta-1}^\eta)} + \|w_k\|_{L^2(G_{\eta-1}^\eta)}\\
&\le C\, M + |Q| \max \limits_{\overline G_{\eta-1}^\eta} w_k\\
&\le C\, M+\alpha' |Q| w_k(\eta,x') \\
& \le C\, M + \alpha' |Q| \delta \min \limits_{\overline Q} u_k(0,x'),
\end{align*}
where $M$ is given by \eqref{M}, $\alpha'>0$ is the constant from the Harnack inequality for $w_k$ and $\delta$ is defined in \eqref{3.5}; $C$ depends on $\Lambda, d, Q$.
We chose $\delta$ such that $\alpha' \delta < 1/2$ and get
\[
\min \limits_{\overline Q} u_k(0,x') \le C\, M.
\]
Note that $\delta$ only depends on $\Lambda, Q$ and $d$.
Now we can simply consider $u_k$ in two cylinders, $G_{-k}^0$ and $G_0^k$, separately to obtain
\begin{align}
\|u_k\|_{L^2(G_N^{N+1})}  + \|\nabla u_k\|_{L^2(G_{-k}^k)} \le C\, M.
\end{align}
The last estimate imply that, up to a subsequence, $u_k$ converges weakly in $H_{loc}^1(\GG)$, as $k \to \infty$, to a solution $u$ of problem in the infinite cylinder \eqref{eq:orig-prob}. Due to Theorem~\ref{PP-Th-semi-inf}, the restrictions of $u(x)$ to the semi-infinite cylinders $G_{-\infty}^0$ and $G_0^\infty$ stabilize at the exponential rate to some constants, as $x_1 \to \mp\infty$.

\bigskip

Let now ${\rm supp} f, {\rm supp}\,\, g \subset G_{-\infty}^{\eta}$. Note that we have chosen $\eta$, which might be large, but it depends only on $Q, \Lambda$ and $d$. As before, we consider auxiliary problem \eqref{eq:fin-cyl-case2}, and the first step is to derive estimates for $u_k$ in $G_\eta^{\eta+1}$. The function $u_k$ solves a homogeneous problem in $G_\eta^k$, and since $\bar{b}^{+} < 0$ then $u_k$ decays exponentially with growing $x_1$, and there exists $\xi>\eta$ such that
\[
 |u_k(\xi, x')| \leq C_0 \, \|u_k\|_{L^\infty(S_\eta)} \, e^{- \gamma \xi} \leq
 C_0 \, \alpha \, e^{- \gamma \xi} \min \limits_{\overline G_\eta^{\eta+1}} u_k(x) < {|Q|\over2} \min \limits_{\overline G_\eta^{\eta+1}} u_k(x).
\]
In $G_{-k}^{\xi}$ the function $u_k$ can be represented as a sum $u_k = v_k + w_k$, where $w_k$ solves a homogeneous problem with induced boundary conditions
\begin{equation}
\label{1.16}
  \left\{
   \begin{array}{lcr}
   \displaystyle{
    A\, w_k =  0,}
    \quad \hfill x \in G_{-k}^\xi,
    \\[1mm]
    \displaystyle{
    B\, w_k = 0,} \quad \hfill x \in \Sigma_{-k}^{\xi},
    \\[1mm]
    \displaystyle{
    w_k(\xi, x') = u_k( \xi, x'), \,\, w_k(-k,x') = 0}, \quad \hfill x' \in Q;
   \end{array}
  \right.
\end{equation}
and $v_k$ is a solution of the nonhomogeneous equation with homogeneous boundary conditions
\begin{equation}
\label{1.17}
  \left\{
   \begin{array}{lcr}
   \displaystyle{
    A\, v_k =  f(x),}
    \quad \hfill x \in G_{-k}^\xi,
    \\[1mm]
    \displaystyle{
    B\, v_k = g(x),} \quad \hfill x \in \Sigma_{-k}^{\xi},
    \\[1mm]
    \displaystyle
    v_k( - k, x') = v_k(\xi,x') = 0, \quad \hfill x' \in Q.
   \end{array}
  \right.
\end{equation}
Using the maximum principle for $w_k$ and estimates in Lemma~\ref{PP-Th-v^k-2} for $v_k$, we get
\[
 \|u_k\|_{L^2(G_N^{N+1})} \le \|v_k\|_{L^2(G_N^{N+1})} + \|w_k\|_{L^2(G_N^{N+1})}
 \le C\, M  + |Q|\|u_k\|_{L^\infty(S_\eta)}, \quad N<\xi,
\]
with the constant $M$ defined by \eqref{M}.

Thus,
\begin{align*}
|Q|\min \limits_{\overline G_\eta^{\eta+1}} u_k(x) \le \|u_k\|_{L^2(G_\eta^{\eta+1})} & \le \|v_k\|_{L^2(G_\eta^{\eta+1})} + \|w_k\|_{L^2(G_\eta^{\eta+1})}
\\
&< {|Q|\over2} \min \limits_{\overline G_\eta^{\eta+1}} u_k(x) + C\, M,
\end{align*}
and, consequently
\begin{equation}
\label{1.18}
 u_k(\xi, x') \le \frac{|Q|}{2} \, \min \limits_{\overline G_\eta^{\eta+1}} u_k(x) \leq C\, M.
\end{equation}
Since $\bar{b}^{+}<0$, by Lemma~\ref{PP-Th-v^k-1} and \eqref{1.18} we have
\[
|u_k(x)| \le C_0 \, \|u_k\|_{L^\infty(S_\xi)} \, e^{-\gamma x_1} \le C\, M \, e^{-\gamma x_1}, \quad C_0, \gamma > 0, \,\, x \in G_\xi^k.
\]
In the cylinder $G_{-k}^\xi$ we have
\[
 \|u_k\|_{L^2(G_N^{N+1})} \le C M, \quad N<\xi.
\]
The elliptic estimates give a local estimate for the gradient of $u_k$:
\[
\|\nabla u_k\|_{L^2(G_N^{N+1})} \le C \|u_k\|_{L^2(G_{N-1}^{N+2})} \le C \, M.
\]
Thus, $u_k$, up to a subsequence, converges weakly in $H_{loc}^1(\GG)$ to a solution $u$ of problem \eqref{eq:orig-prob}. This solution, restricted to the semi-infinite cylinders $G_{-\infty}^0$ and $G_0^{+\infty}$ stabilizes exponentially to some constants $K^\mp$, as $x_1 \to \mp \infty$.

\bigskip

It is left to prove that a solution is unique up to an additive constant. Suppose that  there are two solutions $u_1$ and $u_2$ to problem \eqref{eq:orig-prob} such that
\begin{align*}
u_l & \to K^{+}, \quad x_1 \to + \infty, \quad l =1,2;\\
u_l & \to K_l^{-}, \quad x_1 \to -\infty, \quad l =1,2.
\end{align*}
Then $w=u_1-u_2$ solves the homogeneous problem
$$
\left\{
\begin{array}{lcr}
A \, w = 0, \quad \hfill x \in \GG,
\\[1mm]
B \, w = 0, \quad \hfill x \in \Sigma,
\end{array}
\right.
$$
and
$$
w \to K^{-} = K_1^{-} - K_2^{-} \neq 0, \quad x_1 \to -\infty; \qquad
w \to 0, \quad x_1 \to +\infty.
$$
Let us consider the restriction of $w$ on the half-cylinder $G_{-\infty}^k$, $k \gg 1$. Since $w \to 0$ at exponential rate, as $x_1 \to +\infty$, then
$$
w(k, x') \le C \, e^{-\gamma k}, \quad C, \, \gamma > 0.
$$
Taking into account that $\bar{b}^{+} \le 0$, we see that $w$ converges to a uniquely defined constant $C_w$, as $x_1 \to -\infty$
$$
|C_w| \le \|w\|_{L^\infty(S_{k})} \le C \, e^{-\gamma k}.
$$
Obviously, whatever $K^{-}$ is, one can chose $k_0$ such that for any $k > k_0$
$$
K^{-} > C \, e^{-\gamma k}.
$$
We arrive at contradiction. Note that, by the maximum principle, a solution to a homogeneous problem that decays to zero when $x_1 \to \pm \infty$, is necessarily zero.

Notice that estimates \eqref{eq:est-u-case2}, as well as the exponential stabilization of a solution to constants, remain valid for generic functions $f \in L^2(\GG)$ and $g \in L^2(\Sigma)$ satisfying condition $\bf (H3)$. Theorem~\ref{Th.2} is proved.

\end{proof}


\section{The case $\bar{b}^{+}\ge0$, $\bar{b}^{-}\le0$}
\label{s_4}

In the case $\bar{b}^{+}\ge 0$, $\bar{b}^{-}\le 0$ a bounded solution of problem \eqref{eq:orig-prob} might fail to exist. Like in the Fredholm theorem, the existence of a bounded solution is granted by an orthogonality condition. Namely, problem \eqref{eq:orig-prob} has a bounded solution if and only if the right-hand side in \eqref{eq:orig-prob} is orthogonal to $p(x)\in H_{loc}^1(\GG) \cap C(\overline \GG)$, a unique, up to a multiplicative constant, bounded solution of the adjoint problem
\begin{equation}
\label{eq-p}
\left\{
\begin{array}{lcr}
A^* p(x) = -{\rm div}(a \nabla p) - {\rm div}(b\, p) = 0, \quad \hfill x \in \GG,
\\[2mm]
B^* p(x) = a\nabla p \cdot n + (b\cdot n)\, p = 0, \quad \hfill x \in \Sigma.
\end{array}
\right.
\end{equation}

\noindent

The next statement asserts the existence and describe the qualitative properties of the ground state of the adjoint operator in the infinite cylinder $\GG$. Note that $p$ decays exponentially in a semi-cylinder if  the corresponding effective drift is nonzero and stabilises to a periodic regime in the semi-cylinder where the effective drift is zero.
\begin{theorem}
\label{Th-main-P(x)}
$ $\\
\noindent
(i) Let $\bar{b}^{+}>0$, $\bar{b}^{-}<0$.
There exists a unique, up to a multiplicative constant, positive function $p(x) \in H^1(\GG) \cap C(\overline{\GG})$ solving problem \eqref{eq-p}. Moreover, under the normalization condition
$
\max \limits_{\GG} p(x) \, dx = 1
$
the estimate holds
\begin{equation}
\label{est-p}
p(x) \le Ce^{-\delta |x_1|}, \quad x \in \GG,
\end{equation}
with the constants $\delta>0$ and $C$ depending only on $\Lambda, d, Q$ and $\bar b^\pm$.

$ $\\
\noindent
(ii) Let $\bar{b}^-<0$, $\bar{b}^+=0$.
There exists a unique, up to a multiplicative constant, positive function $p(x) \in H_{loc}^1(\GG) \cap C(\overline{\GG})$ solving problem \eqref{eq-p}. Under the normalization condition
$
\max \limits_{\GG} p(x) \, dx = 1
$
the function $p$ decays exponentially as $x_1 \to -\infty$ and stabilises to a periodic function $p^+$ solving \eqref{eq:p-peridic} when $x_1 \to +\infty$:
\begin{align}
\label{est-p2}
p(x) \le Ce^{\delta x_1}, \quad x \in G_{-\infty}^0, \\
|p - p^+| \to 0, \quad x_1 \to + \infty,
\end{align}
with the constants $\delta>0$ and $C$ depending only on $\Lambda, d, Q$ and $\bar b^\pm$.

$ $\\
\noindent
(iii) Let $\bar{b}^-=\bar{b}^+=0$.
There exists a unique, up to a multiplicative constant, positive function $p(x) \in H_{loc}^1(\GG) \cap C(\overline{\GG})$ solving problem \eqref{eq-p}. Under the normalization condition
$
\max \limits_{\GG} p(x) \, dx = 1
$
the function $p$ stabilises to periodic functions $p^+$ and $p^-$  solving \eqref{eq:p-peridic} when $x_1 \to \pm\infty$ respectively:
\begin{align}
\label{est-p3}
|p - p^\pm| \to 0, \quad x_1 \to \pm \infty.
\end{align}

\end{theorem}
The proof is presented in Section~\ref{sec:proof-p(x)}.

The main result of this section is  the following theorem.
\begin{theorem}
\label{Th.3}
Let conditions $\bf (H1)-(H3)$ be fulfilled, and suppose that $\bar{b}^{+}\ge0$ and $\bar{b}^{-}\le0$. Then problem \eqref{eq:orig-prob} has a bounded solution if and only if
\begin{equation}
\label{compat_cond}
\int \limits_{\GG} f(x) p(x)\, dx + \int \limits_{\Sigma} g(x) p(x) \, d \sigma = 0.
\end{equation}
where $p(x)$ is defined by Theorem~\ref{Th-main-P(x)}.

Moreover, a solution $u(x)$ to problem \eqref{eq:orig-prob} is unique, up to an additive constant, it stabilizes to some constants at infinity:
\[
\|u - K^{-}\|_{L^2(G_{-\infty}^{-n})} +
\|u - K^{+}\|_{L^2(G_n^{+\infty})} \leq C \, (|K^{+}|+|K^-|) \, e^{- \gamma \, n}, \quad \gamma > 0,
\]
and satisfies the estimates
\begin{align}
\label{est-u-Th.3}
\|u\|_{L^2(G_n^{n+1})} \leq C\,( M + |K^{+}| + |K^{-}|), \quad \|\nabla \, u\|_{L^2(\GG)} \leq C\, M,
\\ \nonumber
M= \|(1 + x_1^2)\,f\|_{L^2(\GG)} +
 \|(1 + x_1^2)\,g\|_{L^2(\Sigma)}.
\end{align}
\end{theorem}

The proof is presented in Section~\ref{sec:case-3}.

\subsection{Proof of Theorem~\ref{Th-main-P(x)}.}
\label{sec:proof-p(x)}

In order to prove the existence of a solution to problem \eqref{eq-p}, we consider the following auxiliary problems defined in  growing cylinders:
\begin{equation}
\label{eq-p_k}
\left\{
\begin{array}{lcl}
&A^* \, p^k = 0, \quad \hfill &x \in G_{-k}^k,
\\[1mm]
&B^* \, p^k = 0, \quad \hfill &x \in \partial G_{-k}^k,
\end{array}
\right.
\end{equation}
where $B^*p^k=a\nabla p^k\cdot n+b\cdot np^k$.
Examples~1 and 2 presented in the end of Section~\ref{s_4} give a motivation for the choice of the 
adjoint
Neumann boundary conditions for $p^k$ on the bases $S_{\pm k} = \{\pm k\} \times Q$.
By the Krein-Rutman theorem (see, for example, \cite{KLS}), problems \eqref{eq-p_k} are solvable, and $p^k(x)$ are positive continuous functions in $\overline{G_{-k}^k}$. The solution $p^k$ is unique up to a multiplicative constant. We normalize $p^k$ in such a way that
\begin{align}
\label{eq:norm-cond-p^k}
\max \limits_{G_{-k}^k} p^k = 1.
\end{align}
Due to \eqref{eq:norm-cond-p^k} and the elliptic estimates, $p^k$ is uniformly
in $k$ bounded in $H^1(G_N^{N+1})$  for any $N$ and, thus, the sequence $\{p^k\}$ converges weakly in $H_{loc}^1(\GG)$ to a solution $p$ of \eqref{eq-p}. Our goal is to show that $p$ is positive, that it tends exponentially to zero at infinity in the half-cylinder where the corresponding effective drift is nonzero, and stabilises to a periodic function in the half-cylinder where the effective drift is equal to zero.

\bigskip
(i) Let $\bar b^+ >0$ and $\bar b^- <0$. We will derive upper and local lower bounds for $p^k(x)$ in the right part of the cylinder, $G_1^k$: The left part $G_{-k}^{-1}$ for $\bar b^- <0$ is considered in the same way.

First we show that $p^k(1,x')$ is bounded from below by a positive constant. To this end we factorize $p^k$ with $p^{+}$, a solution to the periodic problem \eqref{eq:p-peridic}, in $G_1^k$:
\[
p^k(x) = p^{+}(x) \, q^k(x),
\]
then $q^k$ solves the problem
\begin{align}
\label{eq:prob-q^k-1}
\left\{
\begin{array}{lcr}
-{\rm div}(a^+ (p^{+})^2 \nabla q^k) + b^+ (p^{+})^2\cdot \nabla q^k = 0, \quad \hfill x \in G_1^k,
\\[1mm]
a^+ (p^{+})^2 \nabla q^k \cdot n = 0, \quad \hfill x \in \Sigma_1^k,
\\[1mm]
\displaystyle
q^k = \frac{p^k}{p^+}, \quad \hfill x \in S_1 \cup S_k.
\end{array}
\right.
\end{align}
Note that, since $p^+>0$, the function $q^k$ is well defined and positive everywhere in $G_1^k$. For \eqref{eq:prob-q^k-1} the maximum principle is valid, and $q^k$ attains its maximum on the bases $S_1\cup S_k$. 

Since $\min p^+ \le p^+ \le \max p^+$, we have
\[
\max_{\overline S_1\cup \overline S_k} q^k=\max_{\overline G_{1}^k} q^k \ge \frac{1}{\max p^+} >0 \quad \Rightarrow \quad
\max_{\overline S_1\cup \overline S_k} p^k \ge \frac{\min p^+}{\max p^+} > 0.
\]
Let us show that $p^k=o(1)$, $k \to \infty$, on $S_k$, or equivalently let us show that we cannot have $q^k \ge \delta >0$ on $S_k$ for large $k$.

Assume that, for a subsequence,  $q^k \ge \delta > 0$ on $S_k$. For  notation simplicity we do not relabel this subsequence. Since $(p^+)^{-1}$ belongs to the kernel of the periodic adjoint operator associated with \eqref{eq:prob-q^k-1}, the effective drift for \eqref{eq:prob-q^k-1} is
\[
\int_Y (a_{1j}^+ (p^+)^2 \partial_j (p^+)^{-1} + b_1^+ (p^+)^2 (p^+)^{-1}) dx = - \bar b^+ <0.
\]
Since $0 < \delta \le q^k \le 1/\min_{\bar Y} p^+$ on $S_k$, by Lemma~\ref{PP-Th-v^k-1} and the comparison principle, $q^k$ is exponentially close to some constant $C_k^\infty$ in  the interior part of $G_1^k$:
\[
|q^k - C_k^\infty| \le C\, \big(e^{-\gamma x_1}+e^{-\gamma(k-x_1)}\big), \quad  \,\, x \in G_1^k,
\]
and
\[
\|\nabla q^k\|_{L^2(G_N^{N+1})} \le C\, \big(e^{-\gamma N}+e^{-\gamma (k-N)}\big), \quad N>1,
\]
where $0<\delta \le C^\infty_k\le 1/\min_{\bar Y} p^+$, and $\gamma>0$ does not depend on $k$.
Consequently, $p^k$ is exponentially close to $C_k^\infty p^+$:
\begin{align}
\label{eq:1}
|p^k - C_k^\infty p^+| & \le C\, \big(e^{-\gamma x_1}+e^{-\gamma(k-x_1)}\big), \quad \gamma>0, \,\, x \in G_1^k,\notag
\\
\|\nabla p^k - C_k^\infty \nabla p^+\|_{L^2(G_N^{N+1})} &\le C\, \big(e^{-\gamma N}+e^{-\gamma (k-N)}\big), \quad N > 1.
\end{align}
Integrating \eqref{eq-p_k} over $G_\xi^k$, $\xi \ge 1$, we get
\[
\int \limits_{S_\xi} (a_{1j}\partial_j p^k + b_1\, p^k) dx' =
\int \limits_{S_k} (a_{1j}\partial_j p^k + b_1\, p^k ) dx' = 0.
\]
Thus
\[
\int \limits_{G_\xi^{\xi+1}} (a_{1j}\partial_j p^k + b_1 \, p^k ) dx = 0
\]
for any $\xi\in[1,k-1]$.
Using \eqref{eq:1} and passing to the limit in the last equality, we obtain $\bar b^+ = 0$,
which contradicts our assumption. Consequently, $q^k\ge C_1 >0$ on $S_1$ and $q^k$ tends to zero on $S_k$, as $k \to +\infty$, with $C_1$ independent of $k$. In view of the bounds for $p^+$, the same holds for $p^k$:
\begin{align}
\label{eq:est-p^k-1}
p^k\ge C_1 >0 \,\, \mbox{on} \,\,S_1; \quad p^k= o(1)\,\, \mbox{on} \,\, S_k, \quad k \to +\infty.
\end{align}
Therefore,
\begin{align}
\label{eq:est-p^k-2}
p^k  \le C \,  \big(e^{-\gamma x_1}+o(1)\big), \quad k\to \infty, \quad x \in G_1^k.
\end{align}
Since $p^k$ is bounded uniformly in $C^\alpha(\overline{G_1^{k}})$, one can pass to the limit in \eqref{eq:est-p^k-1}--\eqref{eq:est-p^k-2}, as $k\to\infty$, on any compact set in $G_1^\infty$  and  obtain the following estimate for $p$ solving \eqref{eq-p}:
\begin{align}
\label{eq:est-p-1}
0\le p \le C\, e^{-\gamma x_1}, \quad x \in G_1^{+\infty}; \quad p \ge C_1 > 0 \,\, \mbox{on} \,\, S_0.
\end{align}
By the elliptic estimates,
\[
\|p\|_{L^2(G_N^{N+1})} + \|\nabla p\|_{L^2(G_N^{N+1})} \le C\, e^{-\gamma N}, \quad N \ge 1.
\]
Summing up in $N$, we obtain a global $H^1(G_1^{+\infty})$ bound for $p$.

If $\bar b^-<0$ then in the same way we get a uniform $H^1(G_{-\infty}^{-1})$ bound for $p$ and
\[
p \le C\, e^{\gamma x_1}, \quad x \in G_{-\infty}^{-1}.
\]
By the normalization condition \eqref{eq:norm-cond-p^k}, the estimate in \eqref{est-p} holds.

The lower bound in \eqref{eq:est-p-1} on $S_0$ and the Harnack inequality implies that $p$ is positive everywhere in $\mathbb G$.

The uniqueness of $p$, up to a multiplicative constant, follows from Theorem~\ref{Th.3}. Indeed, assume there exist two localized functions $p, p_1$ solving \eqref{eq-p}. Both functions satisfy estimate \eqref{est-p}. We can find a pair $(f,g)$ such that the compatibility condition \eqref{compat_cond} is satisfied with $p$ and \eqref{eq:orig-prob} is solvable. But \eqref{compat_cond} is also necessary, so $(f,g)$ should be orthogonal to both functions, which implies that $p$ and $p_1$ are linearly dependent.

\bigskip
(ii) Now we assume that $\bar b^- <0$ and $\bar b^+ = 0$. The exponential decay of $p$ in $G_{-\infty}^0$ follows from (i). The proof of the stabilization to a periodic regime in the right half-cylinder follows the lines of the proof of Lemma 11  in \cite{AlPi-2016}. The idea of the proof is as follows.

At the first step, as before, we factorize a solution $p^k$ of \eqref{eq-p_k}  by $p^+$ and  show that for the solution $q^k = p^k/p^+$ of \eqref{eq:prob-q^k-1} the following estimates hold (see Lemma 12  in \cite{AlPi-2016}):
\begin{align*}
\max_{S_1} q^k \ge \min_{S_k}q^k, \quad
\min_{S_1} q^k \le \max_{S_k}q^k.
\end{align*}
At the second step by the Harnack inequality we get
\begin{align*}
0 < C_0 \le \min_{G_1^k} q^k \le \max_{G_1^k}q^k \le C_0^{-1}.
\end{align*}
Then it follows from Lemma \ref{PP-Th-v^k-1} that $q^k$ is close to a linear function in $G_1^k$, $q^k$ converges as $k \to \infty$ to $q\neq 0$, a solution to the corresponding problem in the semi-infinite cylinder which stabilizes exponentially to a constant when $x_1 \to \infty$. In terms of $p^k$ this means that $p^k$ satisfies the estimate
\begin{align*}
|p^k - p^+| \le C e^{-\gamma x_1},
\end{align*}
 for some $\gamma$ and $C$.  passing to the limit as $k\to \infty$ yields the desired estimate for $p$.

 \bigskip
(iii) The proof follows the lines of that of Lemma 11 in \cite{AlPi-2016}.

Theorem~\ref{Th-main-P(x)} is proved.


\begin{remark}
\label{remark:est-p^k}
We can improve the estimate \eqref{eq:est-p^k-1}, in the case when the corresponding effective drift is non-trivial, and show that
 $p^k \le e^{-\gamma_1 k}$ on $S_{\pm k}$ for some $\gamma_1>0$.

Decomposing $p^k$ into a sum $r^k + s^k$, where
\[
\left\{
\begin{array}{l}
A^* r^k = 0 \quad \mbox{in} \,\, G_1^k,\\
B^* r^k = 0 \quad \mbox{on} \,\, \Sigma_1^k,\\
r^k\big|_{S_1}=p^k\big|_{S_1} >\delta > 0, \quad r^k\big|_{S_k}=0;
\end{array}
\right.
\qquad
\left\{
\begin{array}{l}
A^* s^k = 0 \quad \mbox{in} \,\, G_1^k,\\
B^* s^k = 0 \quad \mbox{on} \,\, \partial G_1^k\\
s^k\big|_{S_1}=0, \quad s^k\big|_{S_k}=p^k\big|_{S_k}.
\end{array}
\right.
\]
Factorizing $r^k$ and $s^k$ by $p^+$ and repeating the argument in the proof above, we show that for some $\gamma >0$ and a constant $K_k^\infty$
\[
|r^k| \le C\, e^{\gamma x_1}, \quad
|s^k - K_k^\infty\, p^+| \le C \,  e^{-\gamma k}.
\]
We compute the flux through $S_{k/2}$:
\begin{align}
 \nonumber
0 &=\int \limits_{S_{k/2}} (a_{1j}\partial_j p^k + b_1\, p^k ) dx' \\
& = \int \limits_{S_{k/2}} (a_{1j}\partial_j r^k + b_1\, r^k ) dx'
+ \int \limits_{S_{k/2}} (a_{1j}\partial_j s^k + b_1\, s^k ) dx' \\ \nonumber
& =O( e^{-\gamma k}) + K_k^\infty \int \limits_{S_{k/2}} (a_{1j}\partial_j p^+ + b_1\, p^+ ) dx'.
\end{align}
Since $\int \limits_{S_{k/2}} (a_{1j}\partial_j p^+ + b_1\, p^+ ) dx' = \bar b^+ /|Q| >0$, we have $K_k^\infty = O(e^{-\gamma_1 k})$, $k \to \infty$. Thus, $p^k$ on $S_k$ is exponentially small. Similar argument gives that $p^k$ is exponentially small on $S_{-k}$. Consequently, there exist $\gamma_0 >0$ such that
\begin{align}
\label{eq:est-p^k-exponential}
0< p^k  \le C \, e^{-\gamma_0 |x_1|}, \quad x \in G_{-k}^k.
\end{align}
\end{remark}

\bigskip

\subsection{Proof of Theorem~\ref{Th.3}.}
\label{sec:case-3}
We turn to the original problem \eqref{eq:orig-prob}.

To prove the existence of a solution we will consider a sequence of auxiliary problems in a growing family of finite cylinders and pass to the limit, as the length of the cylinder goes to infinity. It should be noted that a sequence of auxiliary problems with Dirichlet boundary conditions will not give us any reasonable approximation. This is illustrated in the two examples below.

\textbf{Example 1.}
\textit{
Consider a family of problems
\begin{equation}
\label{1.22}
\left\{
\begin{array}{lcr}
{v_k}'' + b(x) \, v_k' = f(x), \quad x \in (-k,k),
\\[1.5mm]
v_k(-k)=v_k(k)=0,
\end{array}
\right.
\end{equation}
where $b(x)= - {\rm sign}(x)$, $x \in [-k,k]$, and the function $f(x) = \chi_{[-1,1]}$ is a characteristic function of the interval $[-1,1]$. This equation admits an explicit solution:
$$
v_k(x) = \left\{
\begin{array}{lcr}
(e-1)\, (e^{k} - e^{x})/e, \quad x \in (1,k),
\\[1.5mm]
2(1-e) - x + (e-1) e^{k-1} + e^{|x|}, \quad x \in (-1,1),
\\[1.5mm]
(e-1)\, (e^{k} - e^{-x})/e, \quad x \in (-k,-1).
\end{array}
\right.
$$
Although \eqref{1.22} has a unique solution without any conditions on the right-hand side, this solution can approximate, as $k \to \infty$, no bounded function on $\mathbb{R}^1$ because $v_k \to \infty$ at the exponential rate, as $k \to \infty$.}
\\
\textbf{Example 2.}
\textit{
Let us examine a one-dimensional problem with Neumann boundary conditions
\begin{equation}
\label{1.23}
\left\{
\begin{array}{lcr}
{v_k}'' + b(x) \, {v_k}' = f(x), \quad x \in (-k,k),
\\[1.5mm]
v_k'(-k)=v_k'(k)=0,
\end{array}
\right.
\end{equation}
where $b(x) = - {\rm sign} (x)$, $x \in (-k, k)$. We are going to choose the function $f(x)$ in such a way that the compatibility condition for \eqref{1.23} is satisfied. The kernel of the formally adjoint operator is one-dimensional and consists of functions $\{\lambda \,q_k(x)\}$, $\lambda \in \mathbb{R}$, where $q_k$ solves
\begin{equation}
\label{example_p^k}
\left\{
\begin{array}{lcr}
q_k'' - (b(x)\, q_k)' = 0, \quad x \in (-k,k),
\\[1.5mm]
(q_k' - b \, q_k)(\pm k)=0.
\end{array}
\right.
\end{equation}
We normalize $q_k$ by $\int_{G_{-k}^k} q_k(x) \, dx = 1$. Integrating \eqref{example_p^k} we get
$$
q_k(x) = \frac{1}{2} \, (1 - e^{-k})^{-1} \, e^{-|x|}.
$$
The sequence $\{q_k(x)\}$ converges uniformly to $q(x)= e^{- |x|}/2$, as $k \to \infty$, $x \in \mathbb{R}$.
\\
\\
To satisfy the compatibility condition $\int_{-k}^k q_k\, f\, dx = 0$ we take $f(x)$ such that $f(x) = {\rm sign}(x)$ if $x \in (-1,1)$ and $f(x) = 0$ otherwise.  Using the continuity conditions for the solution and the flux density at the points $x =0$ and $x = \pm 1$, and taking into account that the solution is defined up to an additive constant, we obtain
$$
v_k(x) = \left\{
\begin{array}{lcr}
2 e^{-1}, \quad \hfill x \in (-k, -1),
\\[1mm]
-x + e^{-1}(2 - e^{-x}), \quad \hfill x \in (-1,0),
\\[1mm]
-x + e^{-1} \, e^{x}, \quad \hfill x \in (0,1),
\\[1mm]
0, \quad \hfill  x \in (1,k).
\end{array}
\right.
$$
Obviously, $v_k(x)$ converges uniformly to $v(x)$, as $k \to \infty$, where
$$
v(x) = \left\{
\begin{array}{lcr}
2 e^{-1}, \quad \hfill x \in (-\infty, -1),
\\[1mm]
-x + e^{-1}(2 - e^{-x}), \quad \hfill x \in (-1,0),
\\[1mm]
-x + e^{-1} \, e^{x}, \quad \hfill x\in (0,1),
\\[1mm]
0, \quad \hfill x \in (1,\infty).
\end{array}
\right.
$$
is a solution of the following equation in $\mathbb{R}$:
$$
v'' + b(x) v' = f(x), \quad x \in \mathbb{R}.
$$}
These two examples suggest an idea of using Neumann boundary conditions instead of Dirichlet on the bases $S_{-k}, S_k$ in the auxiliary problems.
\bigskip


We proceed with the proof of Theorem~\ref{Th.3}.

The fact that condition \eqref{compat_cond} is necessary for solvability of \eqref{eq:orig-prob} is evident. Indeed, assume that a solution $u(x)$ to problem \eqref{eq:orig-prob} exists. Multiplying equation \eqref{eq:orig-prob} by $p(x)$ defined in Theorem~\ref{Th-main-P(x)} and integrating by parts yields \eqref{compat_cond}.

Let us now prove that \eqref{compat_cond} is sufficient. We consider the following problems in the growing cylinders $G_{-k}^k$:
\begin{equation}
\label{prob_fin_cyl}
\left\{
\begin{array}{lcr}
A \, u^k = f + r^k, \quad \hfill x \in G_{-k}^k,
\\[1mm]
B \, u^k=g, \quad \hfill x \in \Sigma_{-k}^k,
\\[1mm]
B \, u^k = 0, \quad \hfill x \in S_{-k}\cup S_k.
\end{array}
\right.
\end{equation}
The function $r_k$ is introduced to ensure that the compatibility condition is satisfied. It is defined as follows:
\begin{equation}\label{r_k}
\begin{array}{c}
\displaystyle
r^k = - \frac{\int_{\GG} f(x)\chi(G_{-k}^k) p^k(x) \, dx + \int_{\Sigma} g(x) \chi(G_{-k}^k) p^k(x)  \,d\sigma}{\int_{G_{-1}^1}p^k(x) \, dx} \,\, \chi(G_{-1}^1)\\[6mm]
\displaystyle
= - \frac{\int_{\GG} f(x) \big(\chi(G_{-k}^k) p^k(x) - p(x)\big)\, dx + \int_{\Sigma} g(x)
\big(\chi(G_{-k}^k) p^k(x) - p(x)\big)\,d\sigma}{\int_{G_{-1}^1}p^k(x) \, dx} \,\, \chi(G_{-1}^1)
,
\end{array}
\end{equation}
where $\chi(G_{-1}^1)$ is the characteristic function of $G_{-1}^1$.
One can see that $r^k \to 0$, as $k \to \infty$, and
$$
\int \limits_{G_{-k}^k} (f + r^k) p^k \,dx + \int \limits_{\Sigma} g(x) p^k \,d\sigma = 0.
$$
In order to obtain a priori estimates for $u^k$, we proceed as follows:
\begin{itemize}
\item
Estimate $\|u^k\|_{H^1(G_{-1}^1)}$.
\item
Reduce the problem in $G_{-k}^k$ to a problem in $G_{0}^k$ with a Dirichlet boundary condition on $S_0$ and homogeneous Neumann boundary condition on $S_k$.
\item
Obtain a priori estimates for a solution of the last problem in $G_0^k$ (Lemmata~\ref{lemma-D-N-homogen} and \ref{lemma-D-N-nonhomogen}).
\end{itemize}

We will present a detailed proof only for the case when $\bar b^-< 0, \bar b^+>0$ which is more technical compared with the case when the effective drift is zero in one (or both) of the half-cylinders. Namely, in the latter case, if $\bar b^+=0$, one can estimate the $L^2(G_0^k)$ norm of  $\nabla u^k$ directly, without using $\sqrt{p^k}$ as a weight, and as a consequence, we do not need Proposition  \ref{propos-p^k-decrease}.

Let us normalize a solution to \eqref{prob_fin_cyl}  by
\begin{equation}\label{zzzzz}
\int \limits_{G_0^1} u^k(x) \, dx = 0.
\end{equation}
We multiply the equation in \eqref{prob_fin_cyl} by $p^k \, u^k$ and integrate by parts over $G_{-k}^k$
\begin{equation}
\label{1.37}
\int \limits_{G_{-k}^k} p^k \, (a \, \nabla u^k, \nabla u^k) \, dx =
\int \limits_{G_{-k}^k} (f + r^k) \, p^k \, u^k \, dx +
\int \limits_{\Sigma_{-k}^k} g \, p^k \, u^k \, d\sigma.
\end{equation}
Since $p^k$ decay exponentially to zero when $\bar b^-< 0, \bar b^+>0$, we cannot get an estimate for $\nabla u^k$ in the whole $G_{-k}^k$ at once. We can, however, obtain a bound for $\nabla u^k$ is a finite cylinder, where $p^k$ is bounded from below. That is why we keep $p^k$ on the left-hand side of \eqref{1.37} and estimate the norm of $\sqrt{p^k} \nabla u^k$.

Using the mean value theorem and the Schwartz inequality, we obtain
$$
\begin{array}{l}
\displaystyle
\Big( \int \limits_{G_{-k}^k} f \, p^k \, u^k \, dx \Big)^2 \le
\Big( \sum \limits_{n=-k}^{k-1} \int \limits_{G_{n}^{n+1}} \big|f \, p^k \, u^k\big| \, dx \Big)^2
\\[2mm]
\displaystyle
= \Big( \sum \limits_{n=-k}^{k-1} p^k(\tilde{x}_n) \int \limits_{G_{n}^{n+1}} \big|f \, u^k\big| \, dx \Big)^2 \le
\Big( \sum \limits_{n=-k}^{k-1} p^k(\tilde{x}_n) \|f\|_{L^2(G_{n}^{n+1})} \,  \|u^k\|_{L^2(G_{n}^{n+1})}\Big)^2,
\end{array}
$$
where $\tilde{x}_n \in G_{n}^{n+1}$.

Denote by $x_n$ the points in $G_n^{n+1}$ such that
$$
1  = \sum \limits_{n=-k}^k |Q| \, p^k(x_n)= \int \limits_{G_{-k}^k} p^k(x) \, dx.
$$
Then, by the Harnack inequality,
$$
p^k(\tilde{x}_n) \le \alpha \, p^k(x_n),
$$
with the constant $\alpha$ depending only on $\Lambda, d$ and $Q$. Due to the convexity property of the quadratic function, we have
$$
\Big( \int \limits_{G_{-k}^k} f \, p^k \, u^k \, dx \Big)^2 \le
\alpha \, \sum \limits_{n=-k}^{k-1} p^k(x_n) \|f\|_{L^2(G_{n}^{n+1})}^2 \,  \|u^k\|_{L^2(G_{n}^{n+1})}^2.
$$
By the Poincar\'{e} inequality, recalling \eqref{zzzzz}, we get
$$
\|u^k\|_{L^2(G_0^1)} \le C \, \|\nabla u^k\|_{L^2(G_0^1)}.
$$
Using the last bound, one can see that
$$
\|u^k\|_{L^2(G_n^{n+1})}^2 \le C \,(1 +|n|) \|\nabla u^k\|_{L^2(G_0^{n+1})}^2, \quad \forall n>0.
$$
with the constant $C$ independent of $n$. Thus,
$$
\begin{array}{l}
\displaystyle
\Big( \int \limits_{G_{-k}^k} f \, p^k \, u^k \, dx \Big)^2 \le
C \, \sum \limits_{n=-k}^{k-1} p^k(x_n) \|(1 + \sqrt{|x_1|}) f\|_{L^2(G_{n}^{n+1})}^2 \,  \|\nabla u^k\|_{L^2(G_{0}^{n+1})}^2
\\[2mm]
\displaystyle
= C \, \left\{ \sum \limits_{n=0}^{k-1} + \sum \limits_{n=-k}^{-1}\right\}p^k(x_n) \|(1 + \sqrt{|x_1|}) f\|_{L^2(G_{n}^{n+1})}^2 \,  \|\nabla u^k\|_{L^2(G_{0}^{n+1})}^2 \equiv J_1^k + J_2^k.
\end{array}
$$
Let us estimate $J_1^k$: $J_2^k$ is considered in the same way. Obviously,
$$
J_1^k = C \, \sum \limits_{l=0}^{k-1} \sum \limits_{n=l}^{k-1} p^k(x_n) \|(1 + \sqrt{|x_1|}) f\|_{L^2(G_{n}^{n+1})}^2 \,  \|\nabla u^k\|_{L^2(G_{l}^{l+1})}^2.
$$
We would like to move $p^k(x_n)$ under the norm $\|\nabla u^k\|$, and to do this we need to know that the values of $p^k$ do not differ too much. The following statement gives a kind of monotonicity of $p^k$.
\begin{proposition}
\label{propos-p^k-decrease}
For all $y,z\in G_0^k$ such that $|z_1| > |y_1|$, a solution $p^k$ of problem \eqref{prob_fin_cyl} satisfies the estimate
\begin{equation}
\label{p^k-decrease}
p^k(z_1,z') \le \beta \, p^k(y_1, y'),
\end{equation}
with the constant $\beta > 0$ depending only on $\Lambda, d$ and $Q$.
\end{proposition}
\begin{proof}
Representing $p^k(x)$ in $G_0^k$ as a product $P^k = p^{+}(x) \, \zeta^k(x)$, where $p^{+}$ is a solution of \eqref{eq:p-peridic}, we obtain the following equation for $\zeta^k$:
$$
\left\{
\begin{array}{lcr}
- {\rm div}(a (p^{+})^2 \nabla \zeta^k) + b(p^{+})^2\cdot \nabla \zeta^k= 0, \quad \hfill x \in G_0^k,
\\[2mm]
a (p^{+})^2 \nabla \zeta^k\cdot n = 0, \quad \hfill x \in \Sigma_0^k,
\\[2mm]
\zeta^k(0,x')=(p^{+}(0,x'))^{-1} \, p^k(0,x'), \quad \hfill x'\in Q,
\\[2mm]
\zeta^k(k,x')=(p^{+}(k,x'))^{-1} \, p^k(k,x'), \quad \hfill x'\in Q.
\end{array}
\right.
$$
Note that, in view of the exponential decay of $p^k$
$$
\zeta^k (0,x') \ge \delta_0 \, \big(\max \limits_{Q} p^{+}(0,x')\big)^{-1}, \quad
\zeta^k(k,x') \le Ce^{-\gamma_0 k} \, \big(\min \limits_{Q} p^{+}(k,x')\big)^{-1}.
$$
Denote
$$
M^k(x_1)\equiv \max \limits_{x' \in Q} \zeta^k(x_1,x'), \quad
m^k(x_1)\equiv \min \limits_{x' \in Q} \zeta^k(x_1,x').
$$
By the maximum principle,
$$
m^k(k) \le m^k(x_1) \le M^k(x_1) \le M^k(0), \quad x_1 \in (0,k),
$$
and $M^k(x_1)$ decreases on the interval $[0,\hat{x}_1]$ with $\hat{x}_1 = \min \{x: M^k(x) \le M^k(k)\}$. On the interval $[\hat{x}_1, k]$ (which might consist of only one point) we have $M^k(x_1) \le M^k(k)$.
Take $z$ and $y$ such that $z_1 \ge y_1$. Suppose $M^k(y_1) > M^k(k)$. Then, using the Harnack inequality, we obtain
$$
\alpha^{-1} \, m^k(y_1) \ge M^k(y_1) \ge \max \{M^k(z_1), M^k(k)\} \ge M^k(z_1),
$$
where $\alpha$ depends only on $\Lambda, d$ and $Q$.
If $M^k(y_1) \le M^k(k)$, then $M^k(z_1) \le M^k(k)$ for any $z_1 \ge y_1$ and
$$
m^k(y_1) \ge m^k(k) \ge \alpha \, M^k(k) \ge \alpha \, M^k(z_1).
$$
Thus,
$$
\zeta^k(z_1,z') \le \alpha^{-1} \, \zeta^k(y_1,y'), \quad z_1 \ge y_1, \,\, y',z' \in Q.
$$
Similar inequality takes place in $G_{-k}^0$.

To complete the proof it remains to note that, due to the Harnack inequality,
$$
\frac{\max_Q p^{+}}{\min_Q p^{+}} \le c_0.
$$
\end{proof}
Let us turn back to the estimation of $J_1^k$. By Proposition~\ref{propos-p^k-decrease},
$$
J_1^k \le C \, \sum \limits_{l=0}^{k-1} p^k(x_l) \|(1 + \sqrt{|x_1|}) f\|_{L^2(G_{l}^{k})}^2 \,  \|\nabla u^k\|_{L^2(G_{l}^{l+1})}^2,
$$
and applying again the Harnack inequality yields
$$
\begin{array}{l}
\displaystyle
J_1^k \le C \, \sum \limits_{l=0}^{k-1} \|(1 + \sqrt{|x_1|}) f\|_{L^2(G_{l}^{k})}^2 \,  \|\sqrt{p^k} \, \nabla u^k\|_{L^2(G_{l}^{l+1})}^2
\\[2mm]
\displaystyle
\le C \, \|(1 + \sqrt{|x_1|}) f\|_{L^2(G_0^{k})}^2 \, \|\sqrt{p^k} \, \nabla u^k\|_{L^2(G_{0}^{k})}^2.
\end{array}
$$
Similarly,
$$
J_2^k \le C \, \|(1 + \sqrt{|x_1|}) f\|_{L^2(G_{-k}^{0})}^2 \, \|\sqrt{p^k} \, \nabla u^k\|_{L^2(G_{-k}^{0})}^2,
$$
and, thus,
$$
\Big|\int \limits_{G_{0}^{k}} f \,p^k \, u^k \, dx\Big| \le  C \, \|(1 + \sqrt{|x_1|}) f\|_{L^2(G_{-k}^{k})}^2 \, \|\sqrt{p^k} \, \nabla u^k\|_{L^2(G_{-k}^{k})}^2.
$$
In the same way one can show that
$$
\Big|\int \limits_{G_{0}^{k}} r^k \,p^k \, u^k \, dx\Big| \le  C \, |r^k| \, \|\sqrt{p^k} \, \nabla u^k\|_{L^2(G_{-k}^{k})}^2.
$$
$$
\Big|\int \limits_{\Sigma_{0}^{k}} g \,p^k \, u^k \, d\sigma \Big| \le  C \, \|(1 + \sqrt{|x_1|}) g\|_{L^2(\Sigma_{-k}^{k})}^2 \, \|\sqrt{p^k} \, \nabla u^k\|_{L^2(G_{-k}^{k})}^2.
$$
Note that the factor $(1 + |x_1|)$ is not present in the estimate involving $r^k$: The function $r^k$ is supported on $G_{-1}^1$.

In view of \eqref{1.37},
$$
\|\sqrt{p^k} \, \nabla u^k\|_{L^2(G_{-k}^{k})} \le C \, \big( |r^k| + \|(1 + \sqrt{|x_1|}) f\|_{L^2(\GG)} +
\|(1 + \sqrt{|x_1|}) g\|_{L^2(\Sigma)}\big),
$$
and thus
$$
\|\nabla u^k\|_{L^2(G_{-1}^{1})} \le C \, \big( |r^k| + \|(1 + \sqrt{|x_1|}) f\|_{L^2(\GG)} +
\|(1 + \sqrt{|x_1|}) g\|_{L^2(\Sigma)}\big).
$$
Friedrichs' inequality yields
\begin{equation}
\label{1.39}
\begin{array}{l}
\displaystyle
\|u^k\|_{H^1(G_{-1}^{1})} \le C \, \big( |r^k| + \|(1 + \sqrt{|x_1|}) f\|_{L^2(\GG)} +
\|(1 + \sqrt{|x_1|}) g\|_{L^2(\Sigma)}\big)
\\[2mm]
\le C \, \big(\|(1 + \sqrt{|x_1|}) f\|_{L^2(\GG)} +
\|(1 + \sqrt{|x_1|}) g\|_{L^2(\Sigma)}\big).
\end{array}
\end{equation}
In this way
\begin{equation}
\label{1.40}
\|u^k\|_{H^{1/2}(S_0)} \le C \, \big(\|(1 + \sqrt{|x_1|}) f\|_{L^2(\GG)} +
\|(1 + \sqrt{|x_1|}) g\|_{L^2(\Sigma)}\big).
\end{equation}
It remains to obtain estimates for $u^k$ in $G_{0}^k$ being a solution of the problem
\begin{equation}
\label{1.41}
\left\{
\begin{array}{lcr}
A \, u^k = f+r^k, \quad \hfill x \in G_0^k,
\\[1mm]
B \, u^k = g, \quad \hfill x \in \Sigma_0^k,
\\[1mm]
u^k(0,x')=\psi^k(x'), \quad B\, u^k(k,x')=0, \quad \hfill x' \in Q,
\end{array}
\right.
\end{equation}
where $\psi^k(x')=u^k(0, x')$ satisfies the estimate \eqref{1.40}. The estimates for $u^k$ in $G_{-k}^0$ are obtained similarly. We proceed in two steps: At the first step we consider homogeneous problem with nonhomogeneous Dirichlet boundary condition on $S_0$ (Lemma~\ref{lemma-D-N-homogen}); then, at the second step, we study nonhomogeneous problem with zero Dirichlet boundary condition on $S_0$ (Lemma~\ref{lemma-D-N-nonhomogen}).
\begin{lemma}
\label{lemma-D-N-homogen}
For a solution $u^k$ of problem \eqref{1.41} with $f+r^k=g=0$ the following estimates hold:
$$
\begin{array}{c}
\displaystyle
\|u^k\|_{L^\infty(G_1^k)} \le C \, \|\psi^k\|_{H^{1/2}(Q)},
\\[2mm]
\displaystyle
\|\nabla u^k\|_{L^2(G_0^k)} +
\|u^k\|_{L^2(G_N^{N+1})} \le C \, \|\psi^k\|_{H^{1/2}(Q)}, \quad \forall N,
\end{array}
$$
with a constant $C$ independent of $k$.
\end{lemma}
\begin{proof}
In view of the maximum principle, since $B \, u^k(k,x')=0$,
$$
 \|u^k\|_{L^\infty(S_1)} \leq \|u^k\|_{L^\infty(S_{1/2})}.
$$
In the cylinder $G_0^1$ we represent $u^k$ as a sum $v^k + w^k$, where $v^k$ and $w^k$ satisfy the homogeneous equation and homogeneous boundary conditions on $\Sigma_{-k}^k$, $v^k(0,x') = \psi^k(x')$, $v^k(1,x') = w^k(0,x') = 0$, $w^k(1,x') = u^k(1,x')$. Then the function $v^k(x)$ satisfies the following estimate:
$$
 \|v^k\|_{H^1(G_0^1)} + \|v^k\|_{L^\infty(S_{1/2})} \leq C \, \|\psi^k\|_{H^{1/2}(Q)}.
$$
By the strong maximum principle,
$$
 \|w^k\|_{L^\infty(S_{1/2})} \leq \alpha \, \|u^k\|_{L^\infty(S_1)},
$$
where $0< \alpha < 1$, $\alpha$ does not depend on $k$.
In this way we obtain
$$
 \|u^k\|_{L^\infty(S_1)} \leq \|u^k\|_{L^\infty(S_{1/2})} \leq
 \|v^k\|_{L^\infty(S_{1/2})} + \|w^k\|_{L^\infty(S_{1/2})}
 $$
 $$
 \leq C \|\psi^k\|_{H^{1/2}(Q)} + \alpha \, \|u^k\|_{L^\infty(S_1)}, \quad \alpha < 1,
$$
which yields
$$
 \|u^k\|_{L^\infty(S_1)} \leq \frac{C}{1 - \alpha} \, \|\psi^k\|_{H^{1/2}(Q)}, \quad 0 < \alpha < 1.
$$
Applying the maximum principle once more we obtain
\begin{equation}
\label{1.42bis}
 \|u^k\|_{L^\infty(G_1^k)} \leq C \, \|\psi^k\|_{H^{1/2}(Q)},
\end{equation}
with $C$ independent on $k$. It follows from \eqref{1.42bis} that
\begin{equation}
\label{1.42}
 \|u^k\|_{L^{2}(G_N^{N+1})}  + \|\nabla u^k\|_{L^2(G_0^2)}\le C\|\psi^k\|_{H^{1/2}(Q)}, \quad N \geq 0.
\end{equation}
Let us note that \eqref{1.42} is valid in $L^\infty(G_\delta^k)$, for any $\delta > 0$.

We proceed with estimating $\|\nabla u^k\|_{L^2(G_0^k)}$.

Multiplying the equation in \eqref{1.41} by $p^{+} u^k$ and integrating the resulting relation by parts over $G_1^k$ gives
$$
\int \limits_{G_1^k} (a \nabla \, u^k, \nabla \, u^k) \, p \, dx =
\int \limits_{S_1} u^k p^{+} \,\frac{\partial u^k}{\partial n_a} \, dx' -\frac{1}{2} \, \Big\{\int\limits_{S_1} -
\int\limits_{S_k}\Big\} \, (u^k)^2\,(a_{1j}\frac{\partial p^{+}}{\partial x_j} - b_1 p^{+}) \, dx'.
$$

Since both $p^{+}(x)$ and $u^k(x)$ are elements of $H^1(G_1^k) \cap L^\infty(G_1^k)$, $p^{+} \, u^k \in H^1(G_1^2)$ and
$$
\|p^{+} u^k\|_{H^1(G_1^2)}\le C \, \|\psi^k\|_{H^{1/2}(Q)}.
$$
Since $\mathrm{div}\big(a\nabla u^k\big) \in L^2(G_0^k)$ and $\mathrm{div}\big(a\nabla p^{+}-b \, p^{+} \big)=0$, then the normal components
of $(a\nabla u^k\big)$ and $\big(a\nabla p^{+}-b \, p^{+}\big)$ on $S_1$ are well-defined elements of $H^{-1/2}(Q)$ (see \cite{DuvautLions}), and the inequality holds
\begin{equation}
\label{1.43}
\|a_{1j}\partial_{x_j}u^k\|_{H^{-1/2}(Q)}\le C\|\psi^k\|_{H^{1/2}(Q)},\quad \|a_{1j}\partial_{x_j} p^{+}- b_1 p^{+}\|_{H^{-1/2}(Q)}\le C.
\end{equation}

Taking into account \eqref{1.42bis} and \eqref{1.43}, we estimate the integral on the left-hand side as follows
$$
\int \limits_{G_1^k} (a \nabla \, u^k, \nabla \, u^k) \, p \, dx \le C \, \big(
\|\psi^k\|^2_{H^{1/2}(Q)}+ \|\psi^k\|_{H^{1/2}(Q)} \, \|\nabla u^k\|_{L^2(G_1^k)}\big).
$$
Finally
$$
   \|\nabla \, u^k\|_{L^2(G_0^k)}  \leq C\|\psi^k\|_{H^{1/2}(Q)},
$$
where $C$ does not depend on $k$. Lemma~\ref{lemma-D-N-homogen} is proved.
\end{proof}
The next statement deals with the nonhomogeneous equation with zero Dirichlet boundary condition on the base $S_0$ and homogeneous Neumann boundary condition on $S_k$.
\begin{lemma}
\label{lemma-D-N-nonhomogen}
Let $u^k$ be a solution of problem \eqref{1.41} with $\psi^k = 0$. Then the following estimate is valid:
$$
\|u^k\|_{L^2(G_N^{N+1})} +
\|\nabla u^k\|_{L^2(G_0^k)} \le C\,M, \quad \forall N,
$$
with the constant $M$ having the form
$$
M =  \|(1 + x_1^2)f\|_{L^2(G_0^k)} +  \|(1 + x_1^2)g\|_{L^2(\Sigma_0^k)}.
$$
\end{lemma}
\begin{proof}
Let us consider a sequence of auxiliary problems
\begin{equation}
\label{1.44}
 \left\{
 \begin{array}{lcr}
  A\, u_n^k = f_n + r_n^k,  \quad \hfill x \in G_0^k,
  \\[2mm]
  \displaystyle{
 B\, u_n^k = g_n}, \quad \hfill x
  \in \Sigma_0^k,
  \\[2mm]
  \displaystyle{
  {u_n^k}(0, x') = 0, \quad {B \, u_n^k}(k, x') = 0}, \quad \hfill x' \in Q.
 \end{array}
 \right.
\end{equation}
Here $f_n(x) = f(x) \chi(G_n^{n+1})$, $r_n^k(x) = r^k \, \chi(G_n^{n+1})$ and $g_n(x) = g(x) \chi(G_n^{n+1})$, $\chi(G_\alpha^\beta)$ is a characteristic function of $G_\alpha^\beta$. Multiplying the equation in \eqref{1.44} by  $p^k(x)u_n^k(x)$ and
integrating by parts over $G_0^k$ gives
$$
\int \limits_{G_0^k} (a \nabla u_n^k, \nabla u_n^k) \, p^k \, dx =
\int \limits_{G_n^{n+1}} (f_n + r^k) \, p^k \, u_n^k \, dx +
\int \limits_{G_n^{n+1}} g_n \, p^k \, u_n^k \, d\sigma.
$$
Friedrichs' inequality reads
$$
\|u_n^k\|_{L^2(G_n^{n+1})} \le (1 + \sqrt{n}) \, \|\nabla u_n^k\|_{L^2(G_0^{n+1})}.
$$
Then, using the Harnack inequality for $p^k$ we obtain
\begin{align*}
\int \limits_{G_0^{n+1}} (a \nabla u_n^k, \nabla u_n^k) \, p^k \, dx &\le
C \, \min \limits_{x' \in Q} p^k(n+1, x') \, \big(
(1 + \sqrt{n}) \|f_n\|_{L^2(G_n^{n+1})} + |r_n^k| \big.
\\
\big.& + (1 + \sqrt{n}) \|g_n\|_{L^2(\Sigma_n^{n+1})} \big) \, \|\nabla u_n^k\|_{L^2(G_0^{n+1})}.
\end{align*}
Dividing bothe sides of the last inequality by $\min_{Q} p^k(n+1,x')$ and using Proposition~\ref{propos-p^k-decrease} to estimate $p^k$, we obtain
\begin{align*}
\beta \, \Lambda \, \|\nabla u_n^k\|_{L^2(G_0^{n+1})}^2 \le
\Lambda \, \min \limits_{G_0^{n+1}} p^k(x) \, (\min \limits_{Q} p^k(n+1,x'))^{-1} \,
\|\nabla u_n^k\|_{L^2(G_n^{n+1})}^2
\\
\le C \, \big(
(1 + \sqrt{n}) \|f_n\|_{L^2(G_n^{n+1})} + |r_n^k|
+ (1 + \sqrt{n}) \|g_n\|_{L^2(\Sigma_n^{n+1})} \big) \, \|\nabla u_n^k\|_{L^2(G_0^{n+1})}.
\end{align*}
Consequently,
$$
\|\nabla u_n^k\|_{L^2(G_0^{n+1})} \le
C \, \big(
(1 + \sqrt{n}) \|f_n\|_{L^2(G_n^{n+1})} + |r_n^k|  + (1 + \sqrt{n}) \|g_n\|_{L^2(\Sigma_n^{n+1})} \big),
$$
and by the Friedrichs' inequality for $m \le n$
$$
\|u_n^k\|_{L^2(G_m^{m+1})} \le
C \, \big(
(1 + n) \|f_n\|_{L^2(G_n^{n+1})} + (1+\sqrt{n}) |r_n^k|  + (1 + n) \|g_n\|_{L^2(\Sigma_n^{n+1})} \big).
$$
Thus,
$$
\|u_n^k\|_{H^{1/2}(S_{n+1})} \le C \, \big(
(1 + n) \|f_n\|_{L^2(G_n^{n+1})} + (1+\sqrt{n}) |r_n^k|  + (1 + n) \|g_n\|_{L^2(\Sigma_n^{n+1})} \big).
$$
By Lemma~\ref{lemma-D-N-homogen} we have
$$
\|\nabla u_n^k\|_{L^2(G_0^{k})} \le
C \, \big(
\|(1 + x_1) f_n\|_{L^2(G_n^{n+1})} + |r_n^k|  + \|(1 + x_1)g_n\|_{L^2(\Sigma_n^{n+1})} \big),
$$
$$
\|u_n^k\|_{L^2(G_m^{m+1})} \le
C \, \big(
\|(1+x_1) f_n\|_{L^2(G_n^{n+1})} + (1+\sqrt{n}) |r_n^k|  + \|(1+x_1) g_n\|_{L^2(\Sigma_n^{n+1})} \big), \quad \forall m.
$$
Obviously, $u^k = \sum_{n=0}^{k-1}u_n^k$ solves problem \eqref{1.41} with $\psi^k = 0$.

By the Cauchy-Schwarz inequality
\begin{align*}
\sum_{n=0}^{k-1} \|(1 + x_1) f_n\|_{L^2(G_n^{n+1})} & \le
\Big(\sum_{n=0}^{k-1} \frac{1}{x_1^2} \Big)^{1/2} \,
\Big(\sum_{n=0}^{k-1} \|(1 + x_1^2) f_n\|_{L^2(G_n^{n+1})}^2 \Big)^{1/2}
\\
& \le C \,
\|(1 + x_1^2) f\|_{L^2(G_0^{k})}.
\end{align*}
Taking into account that ${\rm supp}( r^k) \subset G_{-1}^1$, we get
$$
\|\nabla u^k\|_{L^2(G_0^k)} \le C \,\big( \|(1 + x_1^2) f\|_{L^2(G_0^{k})} + |r^k| +
\|(1 + x_1^2) g\|_{L^2(\Sigma_0^{k})} \big);
$$
$$
\|u^k\|_{L^2(G_N^{N+1})} \le C \,\big( \|(1 + x_1^2) f\|_{L^2(G_0^{k})} + |r^k| +
\|(1 + x_1^2) g\|_{L^2(\Sigma_0^{k})} \big).
$$
Lemma~\ref{lemma-D-N-nonhomogen} is proved.
\end{proof}
Combining \eqref{1.39}, Lemmata~\ref{lemma-D-N-homogen} and \ref{lemma-D-N-nonhomogen}, one can see that for the solution $u^k$ of problem \eqref{prob_fin_cyl} the estimates hold
\begin{equation}
\label{estimate_u^k}
\|u^k\|_{L^2(G_N^{N+1})} \leq C\, \left( \|(1 + x_1^2) f\|_{L^2(\GG)} + \|(1 + x_1^2) g\|_{L^2(\Sigma)} + |r^k|\right),
\end{equation}
\begin{equation}
\label{estimate_nablau^k}
\|\nabla u^k\|_{L^2(G_{-k}^k)} \leq C\, \left( \|(1 + x_1^2) f\|_{L^2(\GG)} + \|(1 + x_1^2) g\|_{L^2(\Sigma)}+ |r^k|\right)
\end{equation}
with $C$ independent of $k$. Hence, up to a subsequence, $u^k$ converges weakly in the space $H_{loc}^1(\GG)$, as $k \to \infty$, to a solution $u(x)$ of problem \eqref{eq:orig-prob} which satisfies estimates \eqref{est-u-Th.3}.

The stabilization of $u$ to constants at infinity is ensured by Theorem~\ref{PP-Th-semi-inf}. The uniqueness of the solution can be proved in the same way as in Theorem~\ref{Th.2}. $\square$


\section{Main results in the semi-infinite cylinder.}
\label{s_5}
For the readers convenience, in this section we summarize the results obtained in \cite{PaPi} for a solution in a semi-infinite cylinder.

Let $G = (0, \infty) \times Q$ be a semi-infinite cylinder in
$\mathbb{R}^d$ with the axis directed along $x_1$, where $Q$ is a bounded
domain in $\mathbb{R}^{d-1}$ with a Lipschitz boundary
$\partial Q$. The lateral boundary of $G$ is denoted by $\Sigma = (0, + \infty) \times \partial Q$. We study the following boundary-value problem:
\begin{equation}
\label{PP-1_1} \left\{
\begin{array}{lcr}
\displaystyle{
    - {\rm div} \, \left(a(x) \, \nabla \, u(x)\right) + b(x)\cdot \nabla \, u(x) = f,}
    \quad \hfill x \in G,
    \\[1mm]
\displaystyle{
    a\nabla u\cdot n = g,} \quad \hfill x \in \Sigma,
    \\[1mm]
\displaystyle{
    u(0, x') = \varphi(x'),} \quad \hfill x' \in Q.
\end{array}
\right.
\end{equation}
Here $a(x)$ is a $d \times d$ matrix satisfying the uniform ellipticity condition, and $b(x)$ is a vector in $\mathbb{R}^d$, $\varphi(x') \in H^{1/2}(Q)$. The
matrix-valued function $a(x)$ and the vector field $b(x)$ are supposed to be measurable, bounded and periodic in $x_1$ functions. The periodicity of the coefficients can be perturbed in some fixed finite cylinder, this will not affect the result.

Concerning the functions $f$ and $g$ we suppose that $f(x) \in L^2(G)$, $g(x) \in L^2(\Sigma)$, and that these functions decay exponentially as $x_1$ goes to infinity, i.e. for some $\gamma_1 >0$
\begin{equation}
\label{PP-7_2}
\begin{array}{c}
\displaystyle
 \| f \|_{L^2(G_N^{N+1})} +
 \| g \|_{L^2(\Sigma_N^{N+1})} \leq C \, e^{- \gamma_1 \, N}, \quad N>0.
 \end{array}
\end{equation}
Let us introduce an auxiliary function $p(x)$ which belongs to the null space of the adjoint operator
\begin{equation}
 \label{PP-problem_for_p}
 \left\{
 \begin{array}{lcr}
    - {\rm div} (a \nabla \, p) - {\rm div}\, (b \, p) = 0, \quad
    \hfill x \in Y,
    \\[1mm]
    a\nabla p\cdot n - (b\cdot n) \, p = 0, \quad \hfill x \in \partial
    Y.
 \end{array}
 \right.
\end{equation}
and the effective convection
\begin{equation}
\label{PP-5_1}
 \bar{b}_1 = \int \limits_{G_0^1} \left( a_{1 j}(x) \partial_j p(x) +
    b_1(x) p(x) \right) \, dx,
\end{equation}

\begin{theorem}
\label{PP-Th-semi-inf}
$ $

\begin{enumerate}
  \item
  Any bounded solution $u(x)$ of problem (\ref{PP-1_1}) stabilizes to a
  constant at the exponential rate as $x_1 \to \infty$, that is
  $$
   \| u(x) - C_\infty\|_{L^2(G_n^\infty)} \leq C\, M \, e^{- \gamma \, n}, \quad
   \forall n \geq 0,
  $$
for some $C_0 > 0$ and $\gamma > 0$, $M = \|(1+x_1^2)f\|_{L^2(G)} + \|(1+x_1^2)g\|_{L^2(\Sigma)}$;
  \item
  $\bar{b}_1 < 0$ if and only if
  for any $\varphi(x') \in H^{1/2}(Q)$ and for any constant $l \in \mathbb{R}$,
  there exists a bounded solution $u(x)$ of problem (\ref{PP-1_1}) that
  converges to the constant $l$, as $x_1 \to \infty$;

  \item
  $\bar{b}_1 \geq 0$ if and only if
   for every boundary condition $\varphi(x')$ there exists a unique
   constant $m(\varphi)$ such that a bounded solution of
   problem (\ref{PP-1_1}) converges to this constant as $x_1 \to
   \infty$.
\end{enumerate}
\end{theorem}

The existence of a solution to \eqref{PP-1_1} has been proved using auxiliary problems in finite growing cylinders.
Namely, we consider the following problems in $G_0^k$:
\begin{equation}
\label{prob_fin_cyl_general}
\left\{
\begin{array}{lcr}
\displaystyle{
    - {\rm div} \, \left(a(x) \, \nabla \, v^k\right) + b(x)\cdot \nabla \, v^k = f,}
    \quad \hfill x \in G_0^k,
    \\[2mm]
\displaystyle{
    a\nabla v^k\cdot n = g,} \quad \hfill x \in \Sigma_0^k,
    \\[2mm]
\displaystyle{
    v^k(0, x') = \varphi(x'), \quad v^k(k, x') = K} \quad \hfill x' \in Q,
\end{array}
\right.
\end{equation}
where $\varphi(x') \in L^\infty(Q)$, $K$ is a constant.

The following statements characterizes the asymptotic behavior of $v^k$.

\begin{lemma}
\label{PP-Th-v^k-1} \textbf{The case $f=g=0$.}
\begin{enumerate}
\item
If $\bar{b}_1 > 0$ then there exist constants $C_\varphi^\infty, \gamma_0 > 0$ and $\gamma>0$ such that
\begin{equation}
\label{estimate_v^k_b1<0}
|v^k - C_\varphi^\infty| \leq C_0 \, \|\varphi\|_{L^\infty(Q)} \, \left( e^{-\gamma_0 x_1} + e^{- \gamma(k - x_1)} \right) +
C \, K \, e^{-\gamma (k -x_1)}.
\end{equation}
In the case of a constant $\varphi$, $C_\varphi^\infty = \varphi$ and
\begin{align*}
|v^k - \varphi| \leq C_0 \, (|\varphi| + K) \, e^{- \gamma(k - x_1)}.
\end{align*}
\item
If $\bar{b}_1 < 0$ then there exists $\gamma>0$ such that
\begin{equation}
\label{estimate_v^k_b1>0}
|v^k - K| \leq C_0 \, (\|\varphi\|_{L^\infty(Q)} + K) \, e^{-\gamma x_1}.
\end{equation}
\item
If $\bar{b}_1 = 0$ then in $G_0^k$ the function $v^k$ is close to a linear function:
\begin{equation}
\label{estimate_v^k_b1=0}
\left|v^k - \frac{C_\varphi^\infty (k-x_1) + K x_1}{k}\right| \leq C_0 \, \|\varphi\|_{L^\infty(Q)} \, e^{-\gamma_0 x_1} +
\frac{C}{k} \, \left( \|\varphi\|_{L^\infty(Q)} + K \right).
\end{equation}
\end{enumerate}
The constant $C_\varphi^\infty$ is uniquely defined (see Lemma~5.1 in \cite{PaPi}).
\end{lemma}

\begin{lemma}
\label{PP-Th-v^k-2} \textbf{The case $\varphi=K=0$, $f,g\neq 0$.}

Independently of the sign of $\bar b_1$, there exists a constant $C_\infty$ such that
\begin{align}
\|v^k - C_\infty\|_{L^2(G_N^{N+1})} & \le C (\|(1+x_1^2)f\|_{L^2(G)} + \|(1+x_1^2)g\|_{L^2(\Sigma)}),
\\
\|\nabla v^k\|_{L^2(G)} & \le C (\|(1+x_1^2)f\|_{L^2(G)} + \|(1+x_1^2)g\|_{L^2(\Sigma)}),
\end{align}
where $C$ is independent of $k$ and $N$.

\end{lemma}


\end{document}